\documentclass[11pt, a4paper]{amsart}
\pdfoutput=1

\usepackage{transparent, multicol}
\usepackage{soul} 
\usepackage[utf8]{inputenc} 
\usepackage[T1]{fontenc}
\usepackage{lmodern}

\usepackage{amsmath, amssymb, amsthm, mathtools, amscd, verbatim, bbold, amsbsy, bbm, amsfonts, enumitem, color, url, enumitem, tikz, wrapfig}

\usepackage[all]{xy}
\usepackage[english]{babel}
\usepackage{graphicx, color}
\DeclareGraphicsExtensions{.pdf,.jpg}
\usepackage{xcolor}
\usepackage[pdfpagelabels]{hyperref}
\usepackage[Option]{overpic}

\theoremstyle{plain}
\newtheorem{thm}{Theorem}[section]
\newtheorem*{thm*}{Theorem}
\newtheorem{prop}[thm]{Proposition}
\newtheorem{lemma}[thm]{Lemma}
\newtheorem*{lemma*}{Lemma}

\newtheorem{example}[thm]{Example}

\theoremstyle{definition}
\newtheorem{definition}[thm]{Definition} 
\newtheorem*{definition*}{Definition} 

\theoremstyle{remark}
\newtheorem{remark}[thm]{Remark}

\numberwithin{equation}{thm}

\newcommand{\R}{{\mathbb{R}}} 
\newcommand{\Z}{\mathbb{Z}}
\newcommand{\Q}{{\mathbb{Q}}}
\newcommand{\C}{{\mathbb{C}}}
\newcommand{\N}{{\mathbb{N}}}
\newcommand{\F}{\mathbb{F}}
\newcommand{\cO}{\mathcal O}

\newcommand{\define}{\mathrel{\mathop:}=}

\newcommand{\longrightharpoonup}{\relbar\joinrel\rightharpoonup}
\newcommand{\pfold}[1]{\mathbin{\raisebox{-0.1em}{\ensuremath{\overset{#1}{\longrightharpoonup}}}}}
\newcommand{\Shadow}{\mathrm{Sh}} 
\newcommand{\id}{{\bf{1}}} 

\renewcommand{\v}{{\mathrm{v}}}

\newcommand{\type}{\tau} 
\newcommand{\dtype}{\hat{\tau}} 
\newcommand{\fa}{a_0} 
\newcommand{\Cf}{\mathcal{C}_0} 
\newcommand{\aW}{W} 
\newcommand{\sW}{W_0} 

\newcommand{\mvert}{\;\vert\;} 

\newcommand{\conv}{\mathrm{conv}}
\newcommand{\App}{\mathcal{A}}

\newcommand{\Stab}{\operatorname{Stab}}
\newcommand{\Mor}{\operatorname{Mor}}

\usepackage{cleveref}

\begin{document}
	
\title[Shadows in the wild]{Shadows in the wild -\\ folded galleries and their applications}

\author{Petra Schwer}
\address{Otto-von-Guericke University Magdeburg, IAG, Postschließfach 4120, 39016 Magdeburg, Germany}
\email{petra.schwer@ovgu.de}

\date{ \today }
\thanks{}

\begin{abstract}
This survey is about combinatorial objects related to reflection groups and their applications in representation theory and arithmetic geometry. Coxeter groups and folded galleries in Coxeter complexes are introduced in detail and illustrated by examples. Further it is explained how they relate to retractions in Bruhat-Tits buildings and to the geometry of affine flag varieties and affine Grassmannians. The goal is to make these topics accessible to a wide audience.  
\end{abstract}

\maketitle


\section{Introduction}\label{introduction}

Symmetric groups $S_n$, for $n\in\N$, permuting the elements of an $n$-element set may seem to be the simplest possible examples of groups. Simultaneously they are also universal in a sense as every finite group is in fact a subgroup of some large enough symmetric group. 

It may not seem obvious, but symmetric groups are also prime examples of Coxeter groups, that is (abstractions of) reflection groups. Have a look at Figure~\ref{fig:ComplexA2}. The symmetric group $S_3$ acts on the figure by permuting the three red vertices of the yellow triangle. 

\begin{figure}[htb]
	\begin{center}
		\begin{overpic}[width=0.5\textwidth]{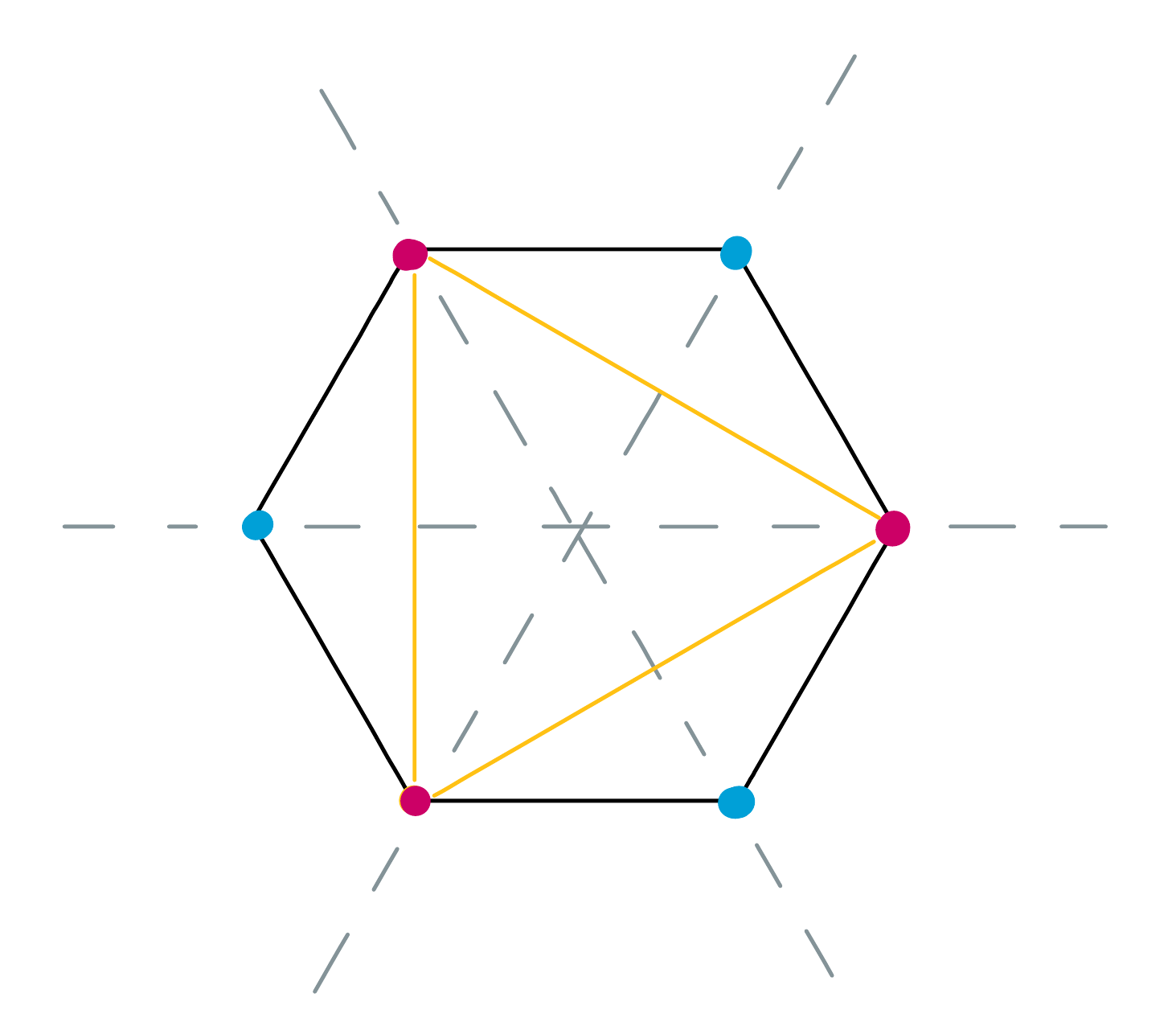}
			\put(45, 70){\tiny{$\id$}}
			\put(70, 55){\tiny{$s$}}
			\put(21, 55){\tiny{$t$}}
			\put(20, 30){\tiny{$st$}}
			\put(70, 30){\tiny{$ts$}}
			\put(45, 15){\tiny{$w_0$}}
		\end{overpic}	
		\caption{Coxeter complex of $S_3$}
		\label{fig:ComplexA2}
	\end{center}
\end{figure}
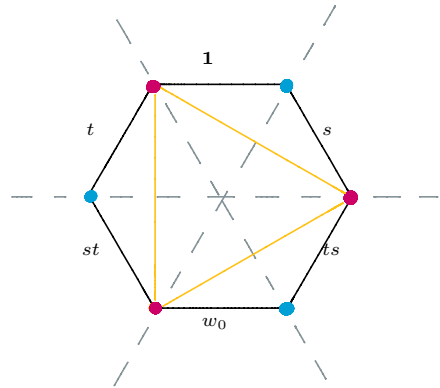

As a reflection group $S_3$ is generated by two reflections along the two diagonal dashed lines. One of these reflections maps the edge labeled $\id$ to the edge labeled $t$ and the other generator the same edge to the one labeled $s$. 
One can see, that $S_3$ hence also acts on the hexagon by color-preserving symmetries. This hexagon is an example of a Coxeter complex which can be associated to any Coxeter group.

Coxeter groups play an important role in many contexts of mathematics. They are the center of study in algebraic combinatorics, where one assigns many interesting discrete structures, combinatorial objects and counting statistics to such groups.
Moreover, Coxeter groups appear as Weyl groups in the study of reductive groups or  Lie groups and in the representation theory of the associated algebras. Many interesting phenomena are governed by the combinatorial structures attached to Coxeter groups. In geometric group theory (Bruhat-Tits) buildings serve as a prime class of nonpositively curved spaces and Coxeter groups as a prime example of CAT(0) groups, the geometry of which has many interesting properties. 

Maybe the most important tool to study a Coxeter group $W$ is their Coxeter complex $\Sigma$, a simplicial complex that has been introduced by Tits in the early 1960s \cite{Tits-Groupes}. This complex encodes how cosets of a specific class of subgroups of $W$ are nested and has the property that the Coxeter group acts simply transitively and cocompactly on that complex. In particular every element in $W$ can be uniquely identified with a maximal simplex in $\Sigma$. 
The group $S_3$ has 6 elements, $\id, s,t,ts,st$ and $sts=w_0$. These six elements can be identified with the six edges of the hexagon as illustrated in Figure~\ref{fig:ComplexA2}.  

In this survey we would like to highlight some recent applications of combinatorial structures associated with affine Coxeter groups. These form an important class of infinite Coxeter groups. A first example is the infinite dihedral group that acts on the number line and is generated by two reflections along adjacent integers as illustrated in Figure~\ref{fig:Dinfty}.   

\begin{figure}[htb]
	\begin{center}
		\includegraphics[width=0.8\textwidth]{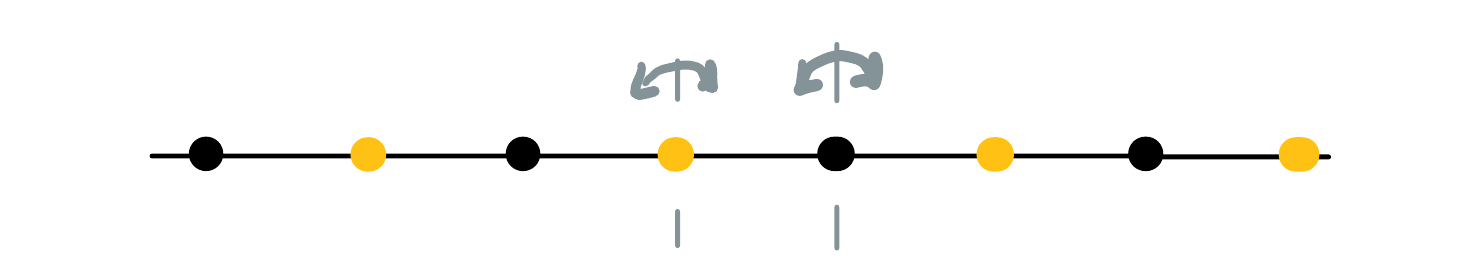}
	\end{center}
	\caption{The infinite dihedral group is generated by two reflections on adjacent integers in $\R$.  }
	\label{fig:Dinfty}
\end{figure}

The key tools for us are so called folded galleries. A gallery in a Coxeter complex is a sequence of maximal cells where subsequent ones share at least a codimension one face. See \ref{def:gallery} for a precise definition and Figure~\ref{fig:folded galleries} for an example. In this example a folded gallery is indicated by a path walking through the maximal simplices in the gallery. 
Folded galleries (and paths) appeared first in the work of Littelmann \cite{Littelmann1, Littelmann2} and Gaussent--Littelmann \cite{GaussentLittelmann} in the context of highest weight representations of semisimple complex Lie algebras. 

Coxeter groups are sometimes Weyl groups associated to a semisimple algebraic group $G$. Gaussent and Littelmann show in \cite{GaussentLittelmann} that the combinatorics of folded galleries is closely linked with the structure of the affine Grassmannian of $G$. There are natural ways to interpret folded galleries and shadows in the Bruhat-Tits building associated with $G$. These buildings are simplicial complexes which are gluings of (possibly infinitely many) copies of a Coxeter complex $\Sigma$ in a symmetric and controlled way such that the group $G$ acts transitively on the copies of $\Sigma$ but also on the set of maximal simplices of the building. As mentioned above, the Coxeter complex associated with the infinite dihedral group is the number line. The corresponding buildings are precisely the simplicial trees without leaves. The Coxeter group in question is the infinite dihedral group. See Figure~\ref{fig:tree_p2} for an example. 

Since their first appearance folded paths and galleries have  been used in numerous places. 
Many of them are highlighted in Ram's survey on folded galleries, the alcove walk model (equivalent to the path model developed by Littelmann) and applications in the context of symmetric functions,  Hecke algebras and structure constants, see \cite{Ram}. 
Moreover, C. Schwer \cite{CSchwer} has used folded galleries to study Hall-Littlewood polynomials. 
Ehrig \cite{Ehrig} showed that MV polytopes can be characterized using folded galleries. 
The famous Saturation conjecture was proven for $SL_n$ by Kapovich and Millson using folded paths in \cite{KM}. 
This gives just a few examples.

The guiding principle will be the following: 
certain double coset intersections of subgroups in a reductive algebraic group with a Bruhat-Tits building can be interpreted in terms of various retractions onto a preferred Coxeter complex inside the building. These retractions are either centered at an alcove, a Weyl chamber or a chimney. See Sections \ref{sec:Littelmann} and \ref{sec:ADLV} for details and Figures~\ref{fig:retraction1} and \ref{fig:retraction_infty} for an illustration of the first two kinds of retractions. In turn the retractions naturally give rise to positively folded galleries. 
We highlight some of these applications in \Cref{sec:applications}. 
Collecting all positively folded galleries of a same type in a single set we may define shadows. These shadows simultaneously capture the image of several galleries under some retractions and have natural interpretations in terms of subgroups of $G$ and also in terms of buildings.  

We will see that retractions and folded galleries explain nonemptiness and dimensions of double cosets in both the affine Grassmannian and affine flag variety. 
This phenomenon will be highlighted in Section \ref{sec:Kostant} where we explain a convexity theorem \cite{Convexity} for reductive groups and buildings analogous to the classical convexity theorem by {Kostant \cite{Kostant}}. 
In Section \ref{sec:ADLV}, which is based on \cite{MST, MST2, MNST}, we will see how the same underlying principle (although with much more involved combinatorial methods and proofs) allows us to show nonemptiness and to compute dimensions of affine Deligne-Lusztig varieties in the affine flag variety as also certain Iwahori double cosets are governed by folded galleries. 

I am by no means claiming completeness, nor originality in writing this survey. My goal is to highlight the combinatorial side of the gallery model for affine flag varieties and affine Grassmanians and hint at its powerful applications. The writing is kept colloquial on purpose in order to keep everything accessible for a wide audience. This also means that if not relevant for the readability not all notions will be formally defined but references will be given. My hope is to pique your curiosity and leave you interested to learn more. Let me know if I was (even remotely) successful.

\subsection*{Organization of this text}

\Cref{sec:CoxeterGroups} lays the foundation for the rest of the survey in introducing Coxeter groups, the associated Coxeter complexes as well as folded galleries. Orientations, positive folds and shadows are introduced in \Cref{sec:orientation}. After that, in Section~\ref{sec:gallerymodel}, the close connection between folded galleries, orbits of subgroups of $G$ and also the Bruhat-Tits buildings and affine Grassmannians is explained. \Cref{sec:Kostant,sec:ADLV} highlight two applications. We close the survey by a quick walk through some more applications and a list of open problems in the last Section~\ref{sec:final}.

 
\section{The main landscape}
\label{sec:CoxeterGroups}

Coxeter groups are a formalization of reflection groups acting on some Euclidean space, sphere or hyperbolic space.
In the 1930s Coxeter studied finite and affine reflection groups and established a particularly nice presentation of them that is encoded in a symmetric matrix with integer entries. In the early 1960 Tits invetigated groups admitting such a representation in a wider context and called these groups Coxeter groups.

In this section we will define Coxeter groups and collect useful facts relevant to the material in this survey. For details and proofs refer to one of the many good textbooks on the topic; for example \cite{BjoernerBrenti, Davis, Bourbaki} or \cite{Humphreys}.

A first example of a well-known class of Coxeter groups are the dihedral groups $D_n$ for $n\in \N\cup\{\infty\}$. For $n\in\N$ these are precisely the symmetry groups of regular $n$-gons. Each of these groups is generated by two reflections. The only infinite group in this class is the infinite dihedral group $D_\infty$ which is the set of symmetries of the number line equipped with a simplicial structure given by the unit intervals, see Figure~\ref{fig:Dinfty}.

A second class of examples are the symmetric groups $S_n$ for $n\in\N$. One may view them as symmetries of an $(n-1)$-simplex acting by permuting the $n$ vertices of the simplex. 
Figure~\ref{fig:ComplexA2} shows the dihedral group $D_6$ which is isomorphic to the symmetric group $S_3$. 

\subsection{Mirror, mirror, on the wall}

Since it is easy to state - here is the formal definition in terms of generators and relations. This presentation is given by a symmetric matrix $(m_{i,j})_{i,j\in I}$ with entries in $\N\cup\{\infty\}$ such that $m_{i,i}=1 < m_{i,j}$ for all $i\neq j \in I$.  

\begin{definition}[Coxeter groups]\label{def:CoxeterGroup}
	A group $W$ with finite generating set $S$ given by the following presentation is a \emph{Coxeter group}: 
	\[
	W=\langle s_i, s_j \in S \mvert s_i^2, (s_is_j)^{m_{i,j}} \rangle, \text{ where } 2\leq m_{i,j}\leq \infty 
	\]
	Here $m_{i,j}=m_{j,i}=\infty$ means that the respective generators do not satisfy any relation. 
	The pair $(W,S)$ is referred to as a \emph{Coxeter system}. Conjugates of $S$ in $W$ are called \emph{reflections}. 
\end{definition}

The defining relations of a Coxeter group are sometimes calles braid relations which does actually have a nice geometric interpretation in terms of braiding strands we will not explain here.  

Most of the time $u, v, w$ denote words in the generators $S$ of $W$ and $[u],[v], [w]$ the associated elements in $W$. The elements in $W$ will be denoted by $x,y,z$ if no defining word is specified. Every subset $S'$ of the generators $S$ defines a \emph{standard parabolic} subgroup $W_{S'}$ of $W$ and each pair $(W_{S'}, S')$ is itself a Coxeter system.

Here are some more examples of finite and infinite Coxeter groups.

\begin{figure}[htb]
	\begin{center}
		\includegraphics[width=0.3\textwidth]{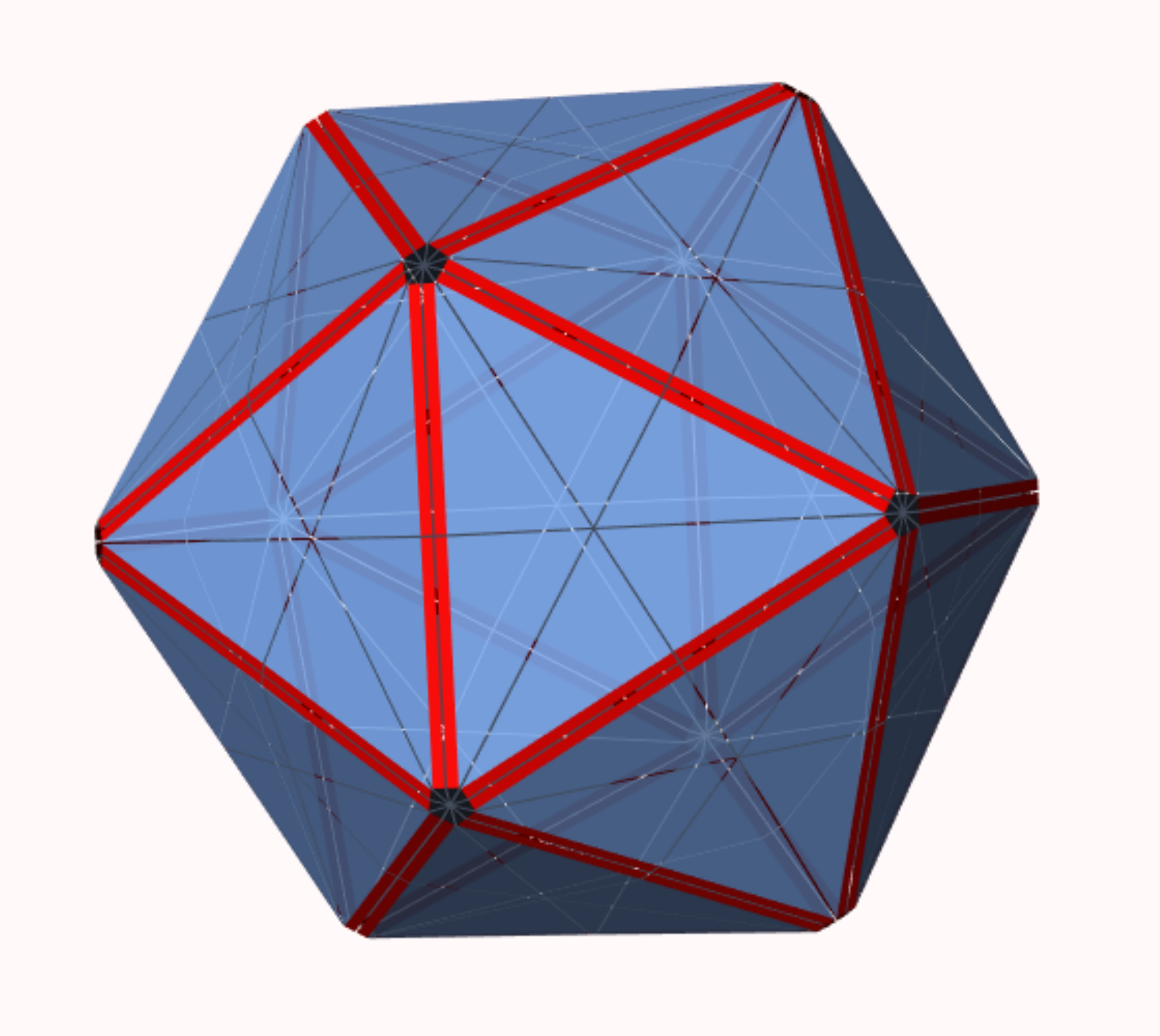} \hspace{2ex}
		\includegraphics[trim = 0 15 0 15, clip,  width=0.3\textwidth]{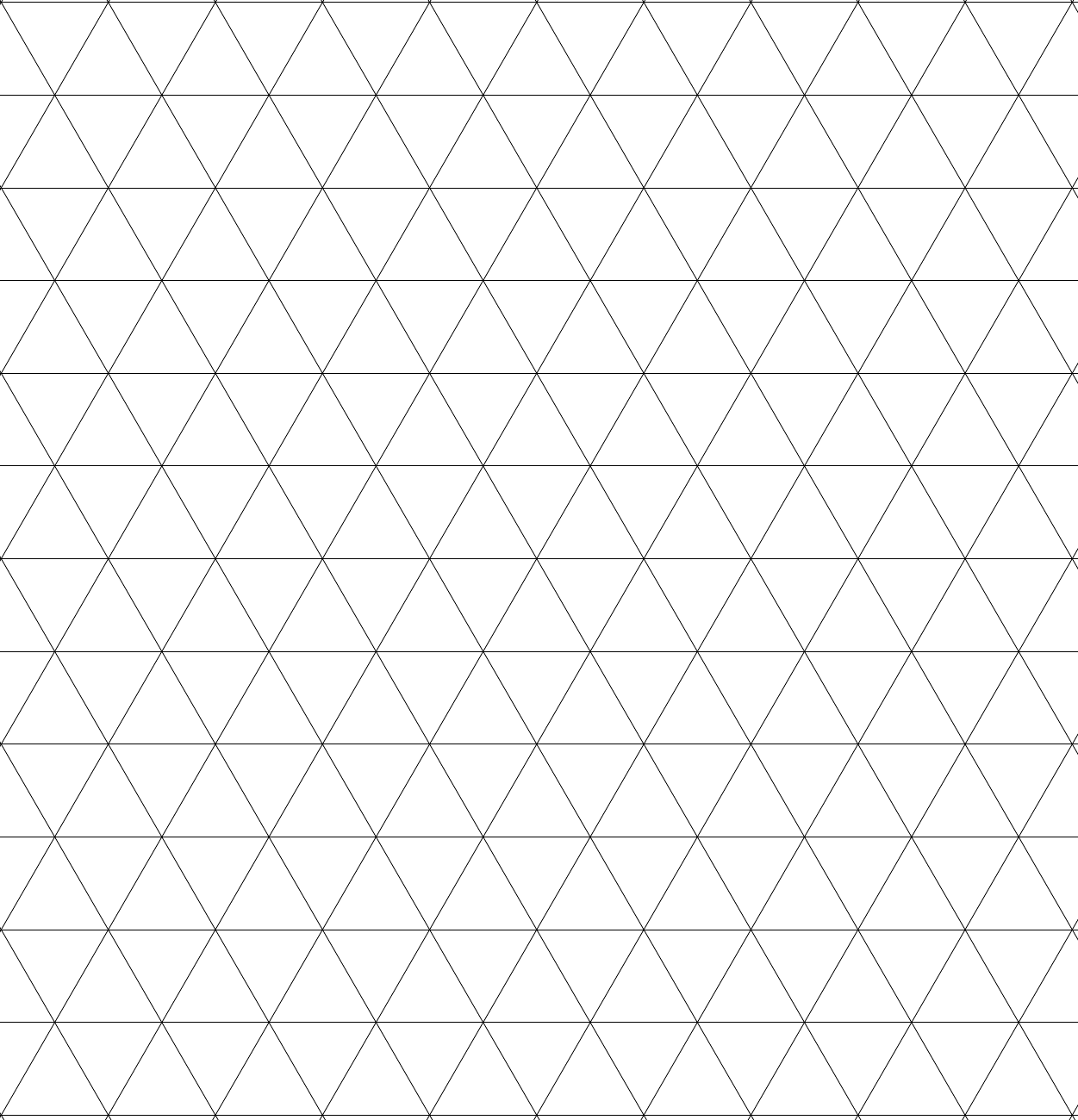} \hspace{2ex}
		\includegraphics[width=0.3\textwidth]{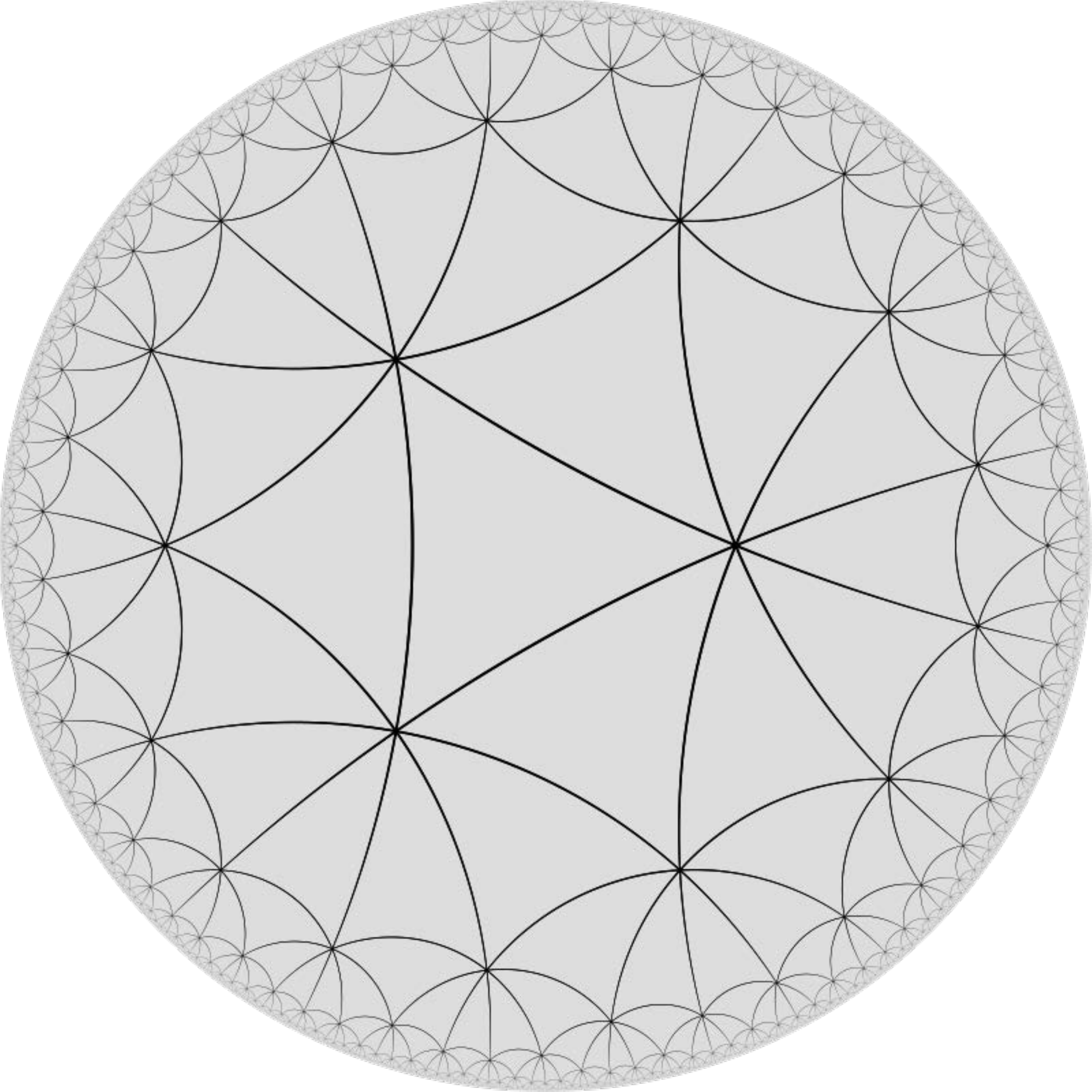} \hspace{2ex}
	\end{center}
	\caption{Three examples of Coxeter complexes. Compare Example~\ref{ex:CoxeterGroups} for details. Picture on the left is taken from Wikipedia.}
	\label{fig:CoxeterGroups}
\end{figure}

\begin{example}\label{ex:CoxeterGroups}
Figure~\ref{fig:CoxeterGroups} shows three Coxeter groups. The one on the left is the symmetry group $W_1$ of the icosahedron. A fundamental domain of the action is a cell of the barycentric subdivision (shown with thin black lines) of the icosahedron. The Coxeter presentation of $W_1$ is 
\[ W_1=\langle s_1, s_2, s_3 \;\vert\; s_i^2,  (s_1s_2)^2, (s_2s_3)^5, (s_3s_1)^3\rangle\]
The group illustrated by the picture in the middle is the one where the pairwise products of generators are all of order three. The presentation is hence given by 
\[W_2=\langle  s_1, s_2, s_3 \;\vert\; s_i^2,   (s_1s_2)^3, (s_2s_3)^3, (s_3s_1)^3\rangle.\]
The group acting simply transitively on this regular tiling of the plane is generated by three reflections bounding any one of the triangles. Compare also Figure~\ref{fig:ComplexA2tilde}. 
Finally on the right we see the tiling of the hyperbolic plane induced by a second kind of infinite reflection group. Its presentation is given by 
\[W_3=\langle  s_1, s_2, s_3 \;\vert\; s_i^2,  (s_1s_2)^4, (s_2s_3)^4, (s_3s_1)^4 \rangle.\]
This group is also generated by three reflections in the hyperbolic plane on which it acts cocompactly. But there is no cocompact isometric action of this group on the euclidean plane. 

Any reflection in such a $W_i$ corresponds to a reflection hyperplane, i.e. a great circle, affine line or bi-infinite geodesic in the respective figures. The collection of all reflection hyperplanes subdivides the sphere, the affine space or disc in disjoint cells and induces a simplicial structure on the space. The \emph{Coxeter complex} of $(W,S)$ is the underlying abstract simplicial complex of these tilings. The standard generating set $S$ corresponds to the hyperplanes going through the faces of a fixed maximal cell of this subdivision. More about this in \Cref{subsec:Coxetercomplex}. 
\end{example}

Any Coxeter group $W$ is uniquely characterized both by a matrix and a graph. The Coxeter matrix $M=M(S,W)$ is the matrix of size $\vert S\vert\times\vert S\vert$ whose entries are the $m_{i,j}$ from the definition of a Coxeter group. The \emph{Coxeter diagram} has vertex set the elements of $S$. Two of them are joined by an edge if they satisfy a relation with $m_{i,j}\neq 2$. In case $m_{i,j}=3$ we draw a single edge, for $m_{i,j}=4$ one draws a double edge and in all other cases the graph has a labeled edge with edge-label the order $m_{i,j}$.  

The Coxeter diagrams for the groups $W_i$, $i=1,2,3$, introduced in Example~\ref{ex:CoxeterGroups} are shown from left to right in Figure~\ref{fig:Dynkin}.

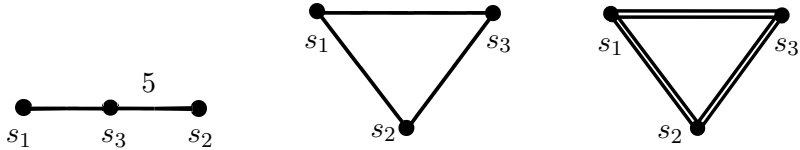
\begin{figure}[h]
		\begin{overpic}[width=0.3\textwidth]{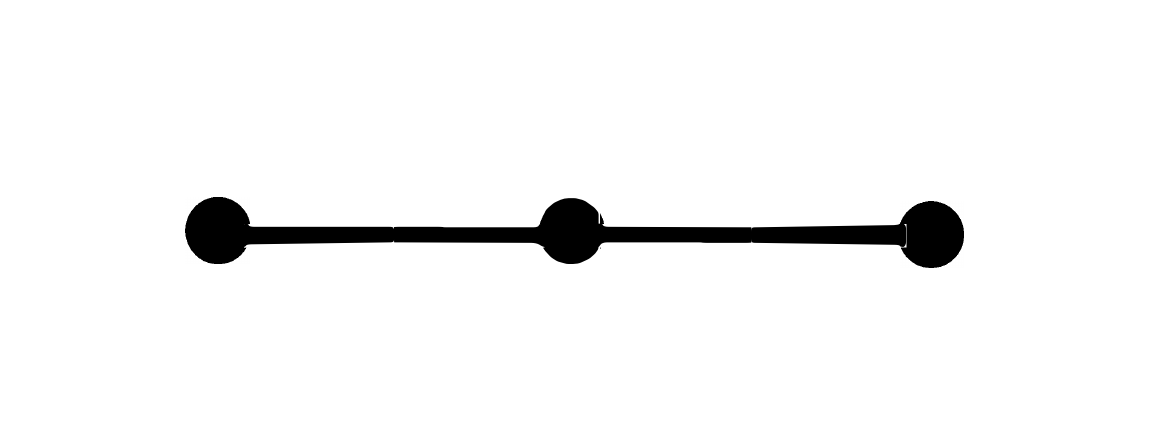}
			\put(17,4){\makebox(0,0)[cb]{$s_1$}}%
			\put(50,4){\makebox(0,0)[cb]{$s_3$}}%
			\put(80,4){\makebox(0,0)[cb]{$s_2$}}%
			\put(62,23){\makebox(0,0)[cb]{$5$}}%
		\end{overpic}
		\begin{overpic}[width=0.3\textwidth]{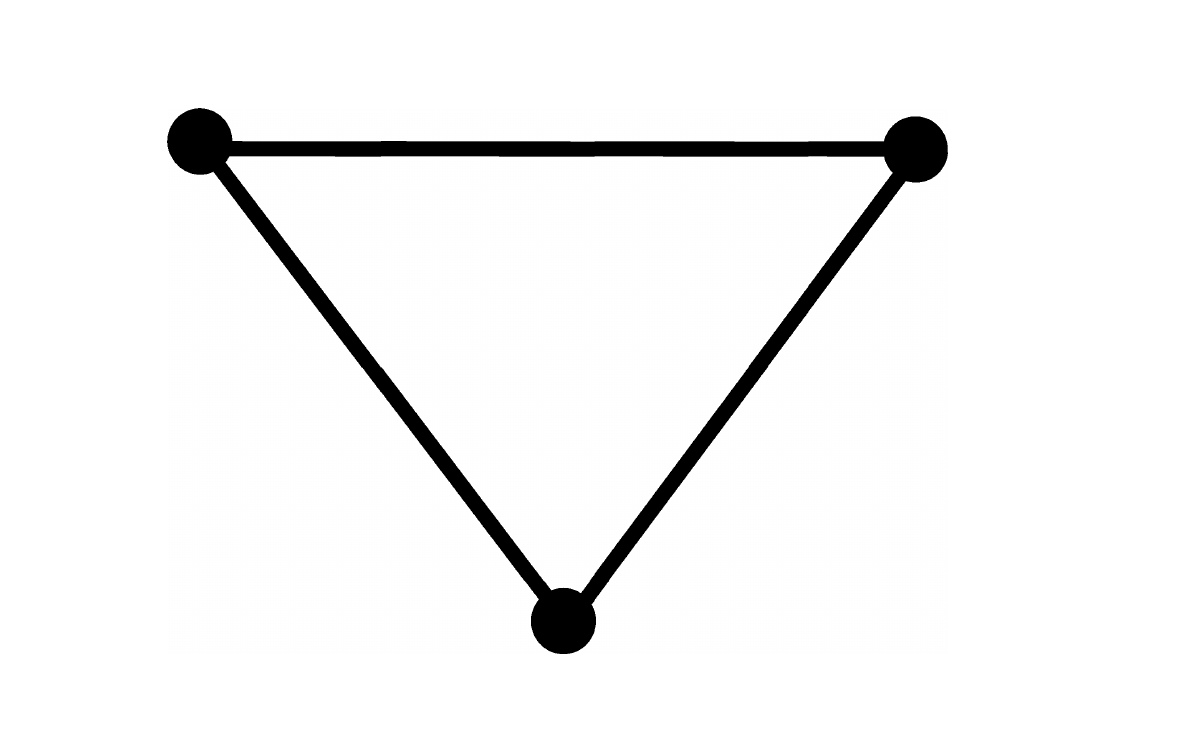}
			\put(17,37){\makebox(0,0)[cb]{$s_1$}}%
			\put(40,6){\makebox(0,0)[cb]{$s_2$}}%
			\put(80,37){\makebox(0,0)[cb]{$s_3$}}%
		\end{overpic}
		\begin{overpic}[width=0.27\textwidth]{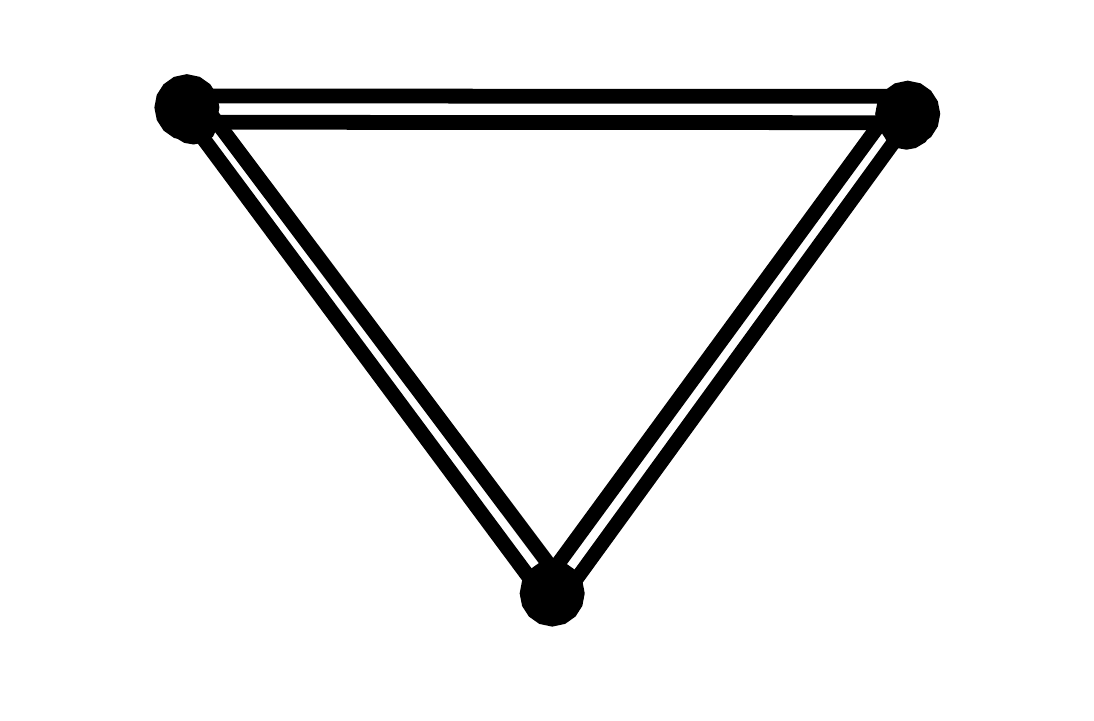}
			\put(17,40){\makebox(0,0)[cb]{$s_1$}}%
			\put(40,6){\makebox(0,0)[cb]{$s_2$}}%
			\put(85,40){\makebox(0,0)[cb]{$s_3$}}%
			
	\end{overpic}
	\caption{Coxeter diagrams for the groups in Example~\ref{ex:CoxeterGroups}. }
	\label{fig:Dynkin}
\end{figure}

In fact Coxeter groups fall into three categories according to the signature of a bilinear form associated with their Coxeter graphs and/or matrices. There are the finite ones acting on some $n$-sphere as $W_1$ in the example above, the infinite ones which are cocompact reflection groups of some euclidean space  just like $W_2$, and the rest. Among the third class are those which admit (not necessarily cocompact) actions on some hyperbolic space, as for example $W_3$ does. See for example \cite{Humphreys} or \cite{Davis} for further details. 

One of the classic problems in group theory is the word problem, namely the problem to decide algorithmically whether two words in the same generators represent the same group element. 
For Coxeter groups this problem has a solution due to Matsumoto \cite{Matsumoto}. Suppose two reduced words represent the same element of a Coxeter group. Matsumoto's theorem then states that the first can be transformed into the second by repeated application of so called \emph{braid moves}, that is by replacing substrings of the form $s_is_js_i\ldots$ and length $m_{i,j}$ by a string of the form $s_js_is_i\ldots$ of the same length. Such a braid move corresponds to the defining relation $(s_is_j)^{m_{i,j}}=1$ in the Coxeter group and plays a crucial role in many proofs and properties of Coxeter groups as we will see in the section on orientations.

\subsection{Turning Coxeter groups into geometric objects}\label{subsec:Coxetercomplex}

Every Coxeter system $(W,S)$ comes with a natural action on a simplicial complex $\Sigma=\Sigma(W,S)$, called the \emph{Coxeter complex}. This complex is always finite dimensional and its maximal simplices are in one-to one correspondence with the elements of $W$. Moreover, $\Sigma$ characterizes a lot of the nice geometric and combinatorial properties of Coxeter groups. As we will see, many algebraic questions can be studied by means of the geometry and combinatorics of this complex. 

Let $W$ be a reflection group acting on some euclidean space, for instance $W_2$ in Example~\ref{ex:CoxeterGroups}. An \emph{alcove} of the tiling induced by this action is (the closure of) a maximal connected component in the complement of the reflection hyperplanes. These are the small triangles in Figure~\ref{fig:CoxeterGroups}. Each of these triangles is a fundamental domain of the action of $W$.

The group itself acts simply transitively on the set of maximal cells of this subdivision. One may hence label them each with a unique element of the group. The cell labeled $\id$ is called \emph{fundamental alcove} and denoted by $\fa$. The hyperplanes bounding $\fa$ then correspond to the elements in $S$. This allows us to label the faces of the cell by their stabilizers in $W$. Here we identify the fundamental alcove $\fa$ with the trivial coset of the trivial group generated by the empty set. All other cells are then labeled by cosets of these stabilizers.  

More formally Coxeter complexes can be introduced as abstract simplicial complexes as follows. Fix a Coxeter system $(W,S)$. 
Subgroups $W_{S'}$ of $W$ generated by a subset $S'\subset S$ are called \emph{standard parabolic subgroup} of $W$. 
In case $S'=\emptyset$ the group $W_{S'}$ is trivial, if $S'=S$ one has that $W_{S'}=W$.  

\begin{definition}[Coxeter complex]\label{def:CoxeterComplex}
Denote by $\Sigma=\Sigma(W,S)$ the set of all left-cosets $xW_{S'}$ of standard parabolic subgroups in $W$. 
The set $\Sigma $ is partially ordered by reverse inclusion. That is, a coset $xW_{S'}$ is less than or equal to some other coset $yW_{S''}$ for some subsets $S', S''\subset S$ and elements $x,y\in W$ if and only if  $yW_{S''}$ is a subset of $xW_{S'}$. More formally 
\[xW_{S'}\leq yW_{S''} \Longleftrightarrow yW_{S''}\subset xW_{S'}.\]
The set $\Sigma$ carries the structure of an abstract simplicial complex and is called \emph{Coxeter complex} of the system $(W,S)$. 
\end{definition}

The maximal simplices in $\Sigma$ are the alcoves and their codimension one faces are called \emph{panels}. The vertices of $\Sigma$ are the cosets of maximal parabolic subgroups, that is subgroups $W'$ generated by subsets $S'=S\setminus\{s\}$ for a single element $s\in S$.  

Note that each panel $p$ corresponds to a coset of a parabolic subgroup of the form $xW_{\{s\}}$ for some $s\in S$. In this case we say $p$ has \emph{type} $s$ and write $\type(p)=s$. We may also define the type of a vertex to be the same as the type of the unique panel in any alcove not containing the vertex. 

Similarly to types of panels we introduce a type for each vertex $\lambda$ in $\Sigma$ by putting $\type(\lambda)=s$ if for an alcove $a$ containing $\lambda$ the codimension one face of $a$ that does not contain $\lambda$ also has type $s$. 
By construction types of panels (and vertices) are invariant under the natural left-action of $W$ on $\sigma$.  

It is worth mentioning that the colored dual graph of $\Sigma$, where one puts a vertex for each alcove and an edge of color $s$ between any pair that shares a panel of type $s$, is exactly the Cayley graph of $W$ with respect to the generating set $S$.  

\begin{remark}
	In the cases considered in this paper, the subdivision of the euclidean space into alcoves yields a geometric realization of the abstractly defined Coxeter complex. This is always true for affine Coxeter groups but does not hold in the general case. 
\end{remark}

We close this section with a small example. 

\begin{example}
	An example for a finite Coxeter complex can be found in Figure~\ref{fig:ComplexA2}. The edges are the alcoves, the vertices are the colored panels where blue represents the generator $s$ and red the generator $t$. 
	Figure~\ref{fig:ComplexA2tilde} shows the Coxeter complex of the infinite Coxeter group $W_2$ we had encountered above. Here alcoves are the triangles and edges the panels. The pictured complex is equivariently colored. The colors of the vertices of the alcove labeled $\id$ are chose such that the opposite panel corresponds to the fix-set of the reflection of the same color. That is, the blue vertex of $\id$ is opposite the codimension one side of $\id$ that is contained in the blue reflection hyperplane.  
\end{example}

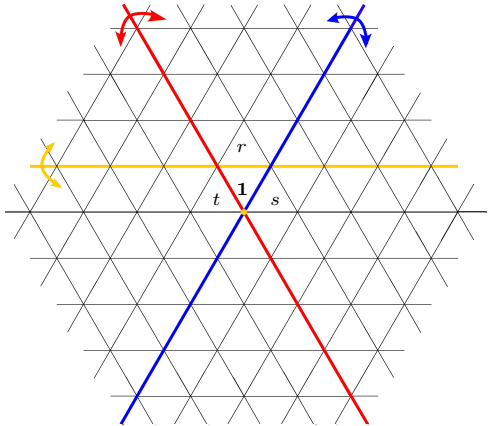
\begin{figure}[htb]
	\begin{center}
		\begin{overpic}[width=0.5\textwidth]{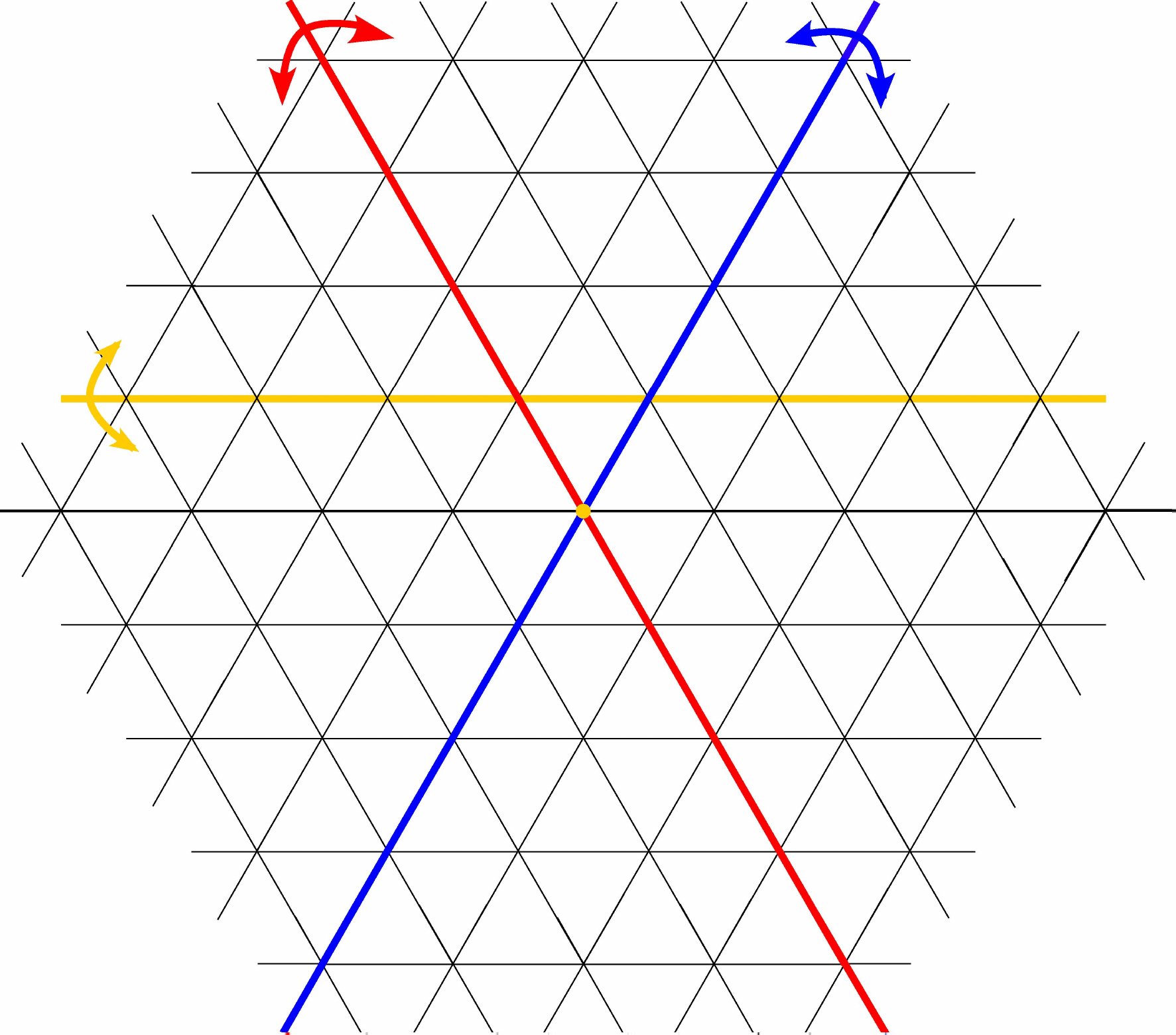}
			\put(48, 48){\tiny{$\id$}}
			\put(48, 57){\tiny{$r$}}
			\put(43, 46){\tiny{$t$}}
			\put(55, 46){\tiny{$s$}}
		\end{overpic}	
		\caption{Coxeter complex of $W_0\cong S_3$}
		\label{fig:ComplexA2tilde}
	\end{center}
\end{figure}

\subsection{Walking around in a Coxeter group}\label{subsec:galleries}

Think of each panel in $\Sigma$ of type $s$ as a door with the letter $s$ written on it. Start in the fundamental alcove $\fa$ labeled $\id$ and walk through a couple of doors remembering the letters written on them along the way. 
You may or may not repeat an alcove or return to the alcove where you started but what you do in any case is reading a word, that is a sequence of letters, in $S$. Each such word represents (not necessarily in a reduced way) an element in your group. 

So not only is every element $x\in W$ represented by an alcove $a_x$ in $\Sigma$ it is also represented by any pathway through doors (we call such a path gallery) that takes you from the fundamental alcove to $a_x$. That is a whole lot of galleries many of which (unnecessarily) go back and forth between adjacent alcoves or take huge detours. 

More formally, what one does is the following. 

\begin{definition}[Galleries]\label{def:gallery}
	A \emph{(combinatorial) gallery} in a Coxeter complex is a sequence
	\[\gamma=(c_0, p_1, c_1, \ldots, p_n, c_n)\]
	of alcoves $c_i$ and panels $p_i$ such that for all $i=1, \ldots, n$ the panel $p_i$ is contained in $c_{i-1}$ and $c_i$. 
	The \emph{length} of the gallery is $n+1$ and we will be calling $c_0$ the \emph{start (alcove)} and $c_n$ the \emph{end (alcove)} of $\gamma$.   
\end{definition}

All galleries in this paper will contain at least one alcove.

Depending on the application it is necessary to modify the definition of a combinatorial gallery and allow for smaller simplices in some places. One such variant we will be using in this survey is the extension a combinatorial gallery on one or both ends by a vertex, i.e. we will consider the following more general form of combinatorial galleries: 

	\[\gamma=(\lambda_0, c_0, p_1, c_1, \ldots, p_n, c_n, \lambda_n). \]

Here $\lambda_0$ is a vertex of $c_0$ and $\lambda_n$ one of $c_n$. In this case we say that $\gamma$ starts in $\lambda_0$ and ends in $\lambda_n$. 
This is for example needed in \cite{GaussentLittelmann, CSchwer} or \cite{Convexity} which will be addressed in Section~\ref{sec:final}. See also \cite{MST} Section 3.2 and Remark 3.13 therein. 

In the above definition nothing prevented us from choosing subsequent alcoves $c_{i-1}$ and $c_{i}$ to be equal. In fact this is critical and makes all the powerful applications possible.

\begin{definition}[Folds]\label{def:folds}
	 A gallery $\gamma=(c_0, p_1, c_1, \ldots, p_n, c_n)$ is \emph{folded at position $i$ (or panel $p_i$)} if $c_i = c_{i-1}$. We say $\gamma$ is \emph{folded} if there exists at least one $i$ such that $\gamma$ is folded at $i$. 
\end{definition}

Typically a (folded) gallery is illustrated by a directed, continuous path in the Coxeter complex that walks through the chambers and panels in the gallery. The direction is indicated by an arrow that points towards the end of the gallery. A bend touching a panel represents a fold at the respective panel and also emphasizes the repeated alcove.
Let's look at an example stolen from \cite[Example 4.4]{Shadows}. 
	
\begin{example}[Folded galleries]\label{ex:folded galleries}
	Figure~\ref{fig:folded galleries} shows two galleries in a type $\tilde{A}_2$ Coxeter complex. The gray gallery starts in the alcove labeled $a$ and ends in alcove $c$. It is not folded but also not minimal. The black gallery also starts in $a$ but ends in $b$. The first bit of the gallery (up to panel $p_4$) agrees with the gray one and both are of the same type. Try to verify this by coloring in the panels. The black gallery has two folds at panels $p_4$ and $p_7$ each of which is illustrated by the peaksy bend touching the respective panel. 
\end{example}

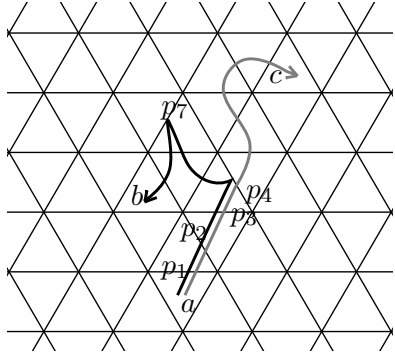
\begin{figure}[htb]
	\begin{overpic}[width=0.4\textwidth]{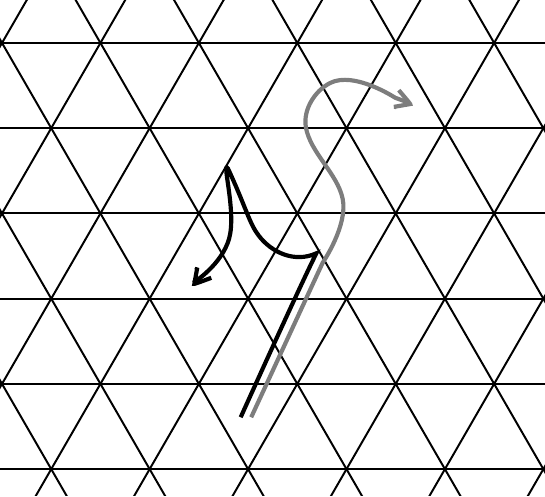}
		\put(45,10){$a$}
		\put(32,38){$b$}
		\put(68,70){$c$}
		\put(40,20){$p_1$}
		\put(45,30){$p_2$}
		\put(58,34){$p_3$}
		\put(62,40){$p_4$}
		\put(40,62){$p_7$}
	\end{overpic}
	
	\caption{This figure shows galleries in a type $\tilde A_2$ Coxeter complex with two folds (black) and no folds (gray). }	
	\label{fig:folded galleries}
\end{figure}

\subsection{The connection with words and elements in $S$ and $W$}\label{subsec:foldsandwords}

The (folded) galleries introduced above have a natural correspondence with (decorated) words in the generators $S$ of $W$, which will now be  explained.  

Recall that a gallery was a sequence $\gamma=(\lambda_0, c_0, p_1, c_1, \ldots, p_n, c_n, \lambda_n). $ 
Each panel $p_i$ in $\gamma$ has a type $s_i\in S$. Listing the types in order and taking the product in order yields a word in $S$, respectively an element in $W$. A word is \emph{reduced} if it is shortest possible among all words representing the same element.  
	
By a \emph{decorated word} we mean a word in $S$ where we put hats on some of its letters. Algebraically these decorations correspond to deletions of the generator in the word. However, for reasons we will discuss later, one wants to remember which letters have been deleted.

\begin{definition}[Type of a gallery]\label{def:gallery type}
	The \emph{type}, denoted by $\type(\gamma)$, of a gallery $\gamma$ of the form 
	\[\gamma=(c_0,  p_1,  c_1, \dots, p_n, c_n), 
	\] 
	is the word in $S$ obtained as follows:
	\[
	\type(\gamma)\define s_{j_1}s_{j_2}\dots s_{j_n},
	\]
	where $s_{j_i}=\type(p_i)$ is the type of the panel $p_i$ for $1 \leq i \leq n$.  
	
	The \emph{decorated type} is the decorated word $\dtype(\gamma)\define s_{j_1}\dots \hat{s}_{j_2} \dots s_{j_n}$ obtained as follows: the $s_{j_i} \in S$ are chosen as above and a hat is put on $s_{j_i}$ in case $c_{i-1}=c_i$ in $\gamma$.  
	By slight abuse of notation we call a letter with a hat a \emph{fold}. 
\end{definition}

Types and decorated types of galleries already carry a lot of information about the galleries themselves. So for example a gallery is a minimal gallery connecting two alcoves if and only if its type is reduced. In this case one also has that type and decorated type agree. Having $\tau=\hat\tau$ yields that the gallery is unfolded but not necessarily minimal. 
	
More properties are listed in the next proposition a proof of which may be found in \cite[Section 4]{Shadows}. 

\begin{prop}
	Let $(W,S)$ be a Coxeter system and fix an alcove $c_0$ in $\Sigma=\Sigma(W,S)$. Then the following hold: 
	\begin{enumerate}
		\item (Reduced) words in $S$ are in bijection with (minimal) unfolded galleries starting in $c_0$. 
		\item Decorated words are in bijection with all galleries starting in $c_0$.  
		\item The end-alcove of a gallery $\gamma$ of decorated type $\dtype(\gamma)$ is the alcove $w.c_0$ where $w$ is the element in $W$ obtained by $\dtype(\gamma)$ after deleting the letters with a hat. 
	\end{enumerate}
\end{prop}	

There are several ways to manipulate a folded gallery. In \cite{MST} we have made crucial use of the Littelmann root operators in \cite{GaussentLittelmann}. In Sections 6, 8.1, 8.3 and 9 of \cite{MST} we moreover introduced several methods to explicitly construct and manipulate galleries via extensions, conjugation or  concatenation. Ram \cite{Ram} as well as Parkinson, Ram and C. Schwer \cite{PRS} also discussed concatenations of folded galleries dressed as alcove walks. Kapovich and Millson studied the closely related Hecke paths and ways to construct them in \cite{KM}. 

Here are two first examples for how to gain new galleries from old ones: 

\begin{itemize}
	\item \textbf{W-action:} The simplicial left-action of $W$ on $\Sigma$ induces an action on the set of all galleries by simultaneously multiplying all alcoves and panels in a gallery on the left. 
	\item \textbf{Explicit (un-) foldings:} Every panel $p_i$ is contained in a unique reflection hyperplane $H_i$ of the complex $\Sigma$. Denote by $r_i$ the reflection along $H_i$ in $W$. One may fold (or unfold) a given gallery $\gamma$ at the panel $p_i$ by reflecting the subgallery $(c_i, p_{i+1},\ldots, p_n, c_n)$ along $H_i$, that is taking its image under $r_i$ according to the just defined $W$ action.     
\end{itemize}

The black gallery in Example~\ref{ex:folded galleries} is obtained from the gray gallery by folding along panel $p_4$ and $p_7$. 

Surprisingly, showing existence or counting the number of folded galleries satisfying some extra conditions (being positively folded or of maximal dimension, notions that will be introduced later) has many interesting algebraic and geometric consequences. I first came across them when proving the buildings-analog of classical Kostant convexity, see \cite{Convexity, Convexity2}. But I am jumping ahead. Let's first talk about those extra conditions in the next section.


\section{Orientations and Shadows}\label{sec:orientation}

Polaris is used in the Northern Hemisphere to orient yourself and to determine which way to turn in order to walk in the direction of north. 
Orientations on Coxeter complexes serve a similar purpose. They are like north stars showing us which fold or crossing faces "north" or "south". 

\subsection{Finding orientation(s)}

Formally an orientation is defined as follows. 

\begin{definition}[Orientations]\label{def:orientation}
	An \emph{orientation} $\phi$ of $\Sigma$ is a map which assigns to a pair of a panel $p$ and an alcove $c$ containing $p$ a value in $\{+1, -1\}$.  We say that $c$ is on the \emph{$\phi$-positive side} (respectively the \emph{$\phi$-negative side}) of $p$ if $\phi(p,c)=+1$ (respectively -1).
\end{definition}

One way to produce an orientation is to take the map $\phi$ to be a constant map which is either $\equiv +1$ or $\equiv -1$. We will refer to these orientations as the \emph{trivial positive/negative orientation}. 


We now define two natural and more interesting classes of orientations. 

\begin{definition}[Alcove orientations] \label{def:alcove orientation}
	Let $c$ be a fixed alcove in $\Sigma$. For any alcove $d$ and any panel $p$ in $d$, let $\phi_{c}(p,d)$ be $+1$ if and only if $d$ and $c$ lie on the same side of the wall spanned by $p$. The resulting orientation $\phi_{c}$ is called the \emph{alcove  orientation towards $c$} or short the \emph{$c$--orientation}. 
\end{definition}

\begin{example}[Alcove orientations]
	Figure~\ref{fig:simplex orientations} shows the alcove orientation towards the alcove labeled $c$ on a type $A_2$ Coxeter complex. In the picture $\phi_c(p,d)=+1$ is indicated with a little "+" on the alcove $d$ close to panel $p$. In this small example hyperplanes are pairs of diametrical vertices. A pair of a panel and alcove containing it is assigned $+1$ (indicated by $+$ in the figure) if the alcove lies on the same side of the hyperplane as $c$ and a $-$ otherwise.  
\end{example}

\begin{figure}[h]
	\begin{minipage}[c]{0.4\textwidth}
		\begin{tikzpicture}
			\def \n {6}
			\def \radius {1.7cm}
			\def \margin {8} 
			
			\draw circle (\radius);
			\node[draw=none, fill=none, inner sep=0pt, minimum width=6pt] at ({360/\n * 2.5}:{\radius+\margin}) {$c$};

			\foreach \s in {1,...,\n}
			{ 
				\node[circle, draw, fill=black!100, inner sep=0pt, minimum width=4pt] at ({360/\n * (\s - 1)}:\radius) {};
			}
			
			\foreach \s in {1,...,3}
			{ 
				\node[draw=none, fill=none] at ({360/\n * (\s - 1)+\margin}:{\radius+\margin}) {+};
				\node[draw=none, fill=none] at ({360/\n * (3+\s - 1)-\margin}:{\radius+\margin}) {+};
				
				\node[draw=none, fill=none] at ({360/\n * (\s - 1)-\margin}:{\radius+\margin}) {--};
				\node[draw=none, fill=none] at ({360/\n * (3+\s - 1)+\margin}:{\radius+\margin}) {--};
			}
		\end{tikzpicture}	
	\end{minipage}
	\caption{The alcove orientation with respect to $c$ on the type $A_2$ Coxeter complex. }
	\label{fig:simplex orientations}
\end{figure}
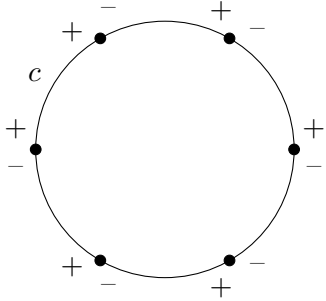

Alcove orientations are wall consistent, that is for all the panels in a same wall and chambers containing the panels on a same side of the wall, the value of the orientation agrees. So for wall consistent orientations each side of a wall is either positive or negative. 

In the next subsection we are leaving the general framework and introduce an orientation available on affine Coxeter complexes only. This orientation is determined by a choice of a chamber at infinity.

\subsection{Orientation far away} 

In this section we restrict to the affine case.  That is, the Coxeter groups in question are affine reflection groups, that is subgroups of the isometry group of a real affine space generated by affine reflections. An affine reflection of $\R^n$ is a reflection at some affine hyperplane preserving the euclidian metric. Throughout this paper an affine Coxeter group is an affine reflection group $W \leq \mathrm{Isom}(\R^n)$ that acts cocompactly with a simplicial fundamental domain. Such a $W$ has a Coxeter generating set $S$ which corresponds to reflections along the hyperplanes spaned by the codimension one faces of its fundamental domain. 
    
To each affine Coxeter group $\aW$  is attached a spherical Coxeter group $\sW$.  Hereby $\sW$ may either be seen as the induced group action on the boundary sphere, as illustrated in Figure~\ref{fig:W-W0} at infinity of $\Sigma$, or as the stabilizer of any special vertex in $\Sigma$. In fact, a vertex is special if and only if its stabilizer is canocially isomorphic to the group acting on the boundary sphere. Some special vertices are shown in bold black in that figure. 

Moreover $\aW$ splits as a semi-direct product of $\sW$ and a translation group $T$ isomorphic to $\Z^n$, for the same $n$ as the dimension of the geometric realization of $\Sigma$. This translation subgroup $T$ has a very natural definition in terms of root systems. See \cite{Humphreys} for further details. 

The identification of $\sW$ with a subgroup of $\aW$ is not
unique. Conjugation by elements of $T$ gives an infinite family of such subgroups.

\begin{figure}[htb]
	\begin{tikzpicture}[>=stealth]
		\draw (0,0) node{
			\begin{tikzpicture}[baseline=2.5cm,scale=.8]
				\draw[draw=none,fill=white!80!blue] (3,4)--(4,4)--(4,3)--(3,4);
				\foreach \x in {1,...,5}{
					\draw (\x,1.8)--(\x,6.2);
					\draw (.8,\x+1)--(5.2,\x+1);
				}
				\foreach \x in {1,...,3}{
					\draw (.8,2*\x+.2)--(2*\x-.8,1.8);
					\draw (.8,2*\x-.2)--(7.2-2*\x,6.2);
				}
				\foreach \x in {1,...,2}{
					\draw (2*\x+.8,6.2)--(5.2,2*\x+1.8);
					\draw (5.2,6.2-2*\x)--(2*\x+.8,1.8);
				}
				\draw[very thick] (.5,6.5)--(5.5,1.5);
				\draw[very thick] (0,4)--(6,4);
				\draw (5,4) node[circle,inner sep=2, fill=black] {};
				\draw (5,6) node[circle,inner sep=2, fill=black] {};
				\draw (5,2) node[circle,inner sep=2, fill=black] {};
				\draw (3,6) node[circle,inner sep=2, fill=black] {};
				\draw (3,4) node[circle,inner sep=2, fill=black] {};
				\draw (3,2) node[circle,inner sep=2, fill=black] {};
				\draw (1,2) node[circle,inner sep=2, fill=black] {};
				\draw (1,4) node[circle,inner sep=2, fill=black] {};
				\draw (1,6) node[circle,inner sep=2, fill=black] {};
			\end{tikzpicture}
		};
		\draw (0,0) circle (4cm);
		\foreach \x in {1,...,8}{
			\draw (45*\x:4cm) node[circle,inner sep=2, fill=black] {}; 
			\draw[dashed] (45*\x:3.25cm)--(45*\x:4.5cm);
		}
		\draw (0.3,0.24) node[inner sep=1,left] {$\lambda$};
		\draw[dashed,->] (.8,-.24) node[inner sep=1,left] {$x$} --
		 (-20:4.5cm) node[below right,inner sep=1] {$p(x)$}; 
	\end{tikzpicture}
	\caption{Elements of $\aW$ can be put in bijection with alcoves in
		tessellated plane and the orbit of $x$ under translations in $\aW$
		is illustrated with the fat vertices of these alcoves.  Elements
		of $\sW$ are in bijection with maximal simplices on the boundary sphere.
		The projection map $p$ maps the element $x$, represented by an alcove (in this picture also labeled $x$) to a chamber labeled $p(x)$ at
		infinity that points ``in the same direction''. This figure is taken from \cite{reflectionlength} where more details are provided.}
	\label{fig:W-W0}
	\label{fig:projection}
\end{figure}
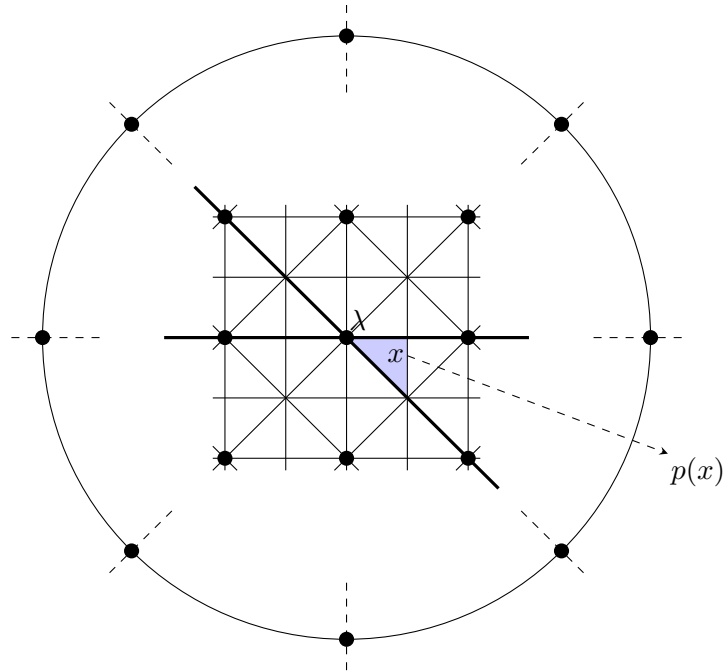

	As we have already explained in Section~\ref{sec:CoxeterGroups} the choice of a vertex $\lambda_0$ as origin and a fundamental alcove $\fa$ having $\lambda_0$ as a vertex determines a bijection between elements of $W$ and alcoves. Here the fundamental alcove $\fa$ corresponds to the identity and every element $x\in W$ corresponds to the image of
	$\fa$ under $x$.
	
\begin{remark}
The visual boundary of $\Sigma$ is a sphere whose simplicial structure is induced by the
tessellation of the plane. A parallelism class of hyperplanes in $\Sigma$
corresponds to a reflection hyperplane, i.e., a a great circle, in the
sphere $\partial\Sigma$. These hyperplanes at infinity induce a simplicial tiling of the sphere. The connected components in the complement of the hyperplanes in $\partial\Sigma$ are called \emph{chambers}.  
Each chamber corresponds to a parallel  class of simplicial cones in $\Sigma$ called  \emph{Weyl chambers}. The 
chambers in $\partial\Sigma$ as well as the alcoves (and Weyl chambers) in $\Sigma$ based at a special vertex are in bijection with the elements of the spherical Weyl group $\sW$. These bijections can be chosen in a compatible way so that the chamber labeled $\id$ in $\partial \Sigma$ corresponds to the \emph{fundamental Weyl chamber $\Cf$} based at the origin containing the fundamental alcove $\fa$ labeled $\id$ in $\Sigma$. 

This arises naturally from the following interpretation of the elements $w\in \aW$. Such a $w$ is an isometric automorphism of $\Sigma$ that
obviously maps parallel rays to parallel rays. Hence it induces an
isometry of (the tessellation of) $\partial\Sigma$ where translations on $\Sigma$
induce the identity on $\partial\Sigma$.

This is indicated by the dotted arrow in
Figure~\ref{fig:projection}. One can think of this as ``walking to
infinity in the direction of $x$ at $\lambda$''. 
\end{remark}

One can use the chambers at infinity to define orientations on an affine Coxeter complex
The following definition is equivalent to the one provided via induced orientations in \cite{Shadows}. This is the viewpoint we had taken in \cite{MST}.

\begin{definition}[Weyl chamber orientations]\label{def:Weyl chamber orientation}
	Suppose $\Sigma$ is an affine Coxeter complex with boundary $\partial\Sigma$ and fix a chamber $\sigma \in \partial\Sigma$. 
	For any alcove $c$ and any panel $p\subset c$ in $\Sigma$, let $H$ be the affine hyperplane containing $p$. Define $\phi_\sigma(p,c)$ to be $+1$ if $\sigma$ has a representative cone $C_\sigma\subset\Sigma$ which lies on the same side of $H$ as $c$. The resulting orientation  $\phi_\sigma$ is called \emph{Weyl chamber orientation with respect to $\sigma$}.  
\end{definition}

Weyl chamber orientations have many nice properties. For one they are wall consistent. On the other hand they are braid-invariant, which means that a positively folded gallery remains positively folded in case a substring is replaced by a substring equivalent to it under braid moves. We'll soon see that these orientations also naturally arise in the setting of affine (Bruhat-Tits) buildings. 
  
Now that we have well behaved orientations in place one may talk about galleries being well behaved with respect to a given such orientation. This is the focus of the next subsection.

\subsection{Shadows in Coxeter complexes} 

Folded galleries have the property that some of the alcoves they visit along the way are repeated along some panel. Fixing an orientation one can decide whether with respect to that orientation a repeated alcove $c_i=c_{i-1}$ is on the positive or negative side of the shared panel $p_i$. According to the value of the orientation on that pair such a fold is called \emph{positive}, respectively \emph{negative}. 

\begin{definition}[Positively folded]\label{def:positive fold}
	Let $\phi$ be an orientation. A combinatorial gallery $\gamma=(\lambda_0, c_0, p_1, c_1, \ldots, p_n, c_n, \lambda_n)$ is \emph{$\phi$-positively (negatively) folded} if for all $i$ either $c_{i-1}\neq c_i$ or $\phi(p_i,c_i)=+1$ (respectively, $-1$). 
\end{definition}

While a combinatorial gallery of length $n$ may have $n$ folds the number of positive folds is uniformly bounded above. For a gallery of arbitrary length in an affine Coxeter complex the uniform bound is equal to the length of the longest element in the associated spherical Weyl group. See \cite{GaussentLittelmann, Shadows} or \cite{MST} for a proof. 

In order to gain a better understanding of the behavior of positively folded galleries we consider their shadows with respect to a fixed orientation. This notion was first introduced in \cite{Shadows} and studied further in \cite{MNST}. 
	
Let $\phi$ be a braid-invariant orientation of $(W,S)$, fix $x\in W$ and let $w$ be any reduced word in $S$ such that $[w]=x$. 
One says that \emph{$x$ $\phi$-positively folds onto $y\in W$} if there exists a gallery $\gamma$ of type $w$ starting in $\fa$ and ending in $y$ that is positively folded with respect to $\phi$. This is written as $x\pfold{\phi} y$.

\begin{definition}[Shadows]\label{def:braid:invariant}
	With $\phi,x,w$ as chosen above the \emph{Shadow of $x$ with respect to $\phi$} is given by
	\[\Shadow_\phi(x)=\{y\in W\mvert x\pfold{\phi} y\}.\]
\end{definition}
 
Shadows as defined here are well-defined for any braid-invariant orientation. In \cite{Shadows} we introduce shadows of words which may also be considered for orientations that are not braid-invariant. 

We will now prove that shadows are a natural generalization of intervals in Bruhat order.  

Bruhat order can be defined by the following subword property. See for example \cite[Thm 2.2.2]{BjoernerBrenti}.  Let $w=s_1s_2\ldots s_n$ be a reduced expression for $x=[w]$ and let $y\in W$. Then an element $y$ is less than or equal to some $x$ in Bruhat order if and only if the word $w$ representing  $x$ has a subword that is a reduced word for $y$. 
We include the proposition that makes this connection precise and also provide its proof, which is directly taken from \cite[Prop.6.8]{Shadows}. 

\begin{prop}[Bruhat order and shadows]\label{prop:bruhat_by_folding}
	Let $\phi_+$ be the trivial positive orientation and let $\phi_\id$ be the alcove orientation towards $\id$.  For any pair of elements $x,y \in W$ one has 
	\[
	x\geq y \; \Longleftrightarrow \; x \pfold{\phi_+} y \; \Longleftrightarrow\;  x \pfold{\phi_\id} y . 
	\]
	In particular 
	\[\Shadow_{\phi_+}(x)=\Shadow_{\phi_\id}(x)=[\id,x].\] 
\end{prop}
\begin{proof}
	The first equivalence is a direct consequence of the defining subword property: reduced expressions are in bijection with minimal galleries and a subword corresponds to a decorated word where there are hats on all of the letters missing in the subword. Such a decorated word corresponds to a folded gallery all of whose folds are automatically positive with respect to the trivial positive orientation $\phi_+$. 
	
	It is also easy to see that $x \pfold{\phi_\id} y$ implies $\; x \pfold{\phi_+} y$, since any $\phi_\id$-positive folding of some gallery is also $\phi_+$-positive.
		
	To show that $x \pfold{\phi_+} y$ implies $x \pfold{\phi_\id} y$ let $w$ be a reduced expression for $x$, $\gamma$ a minimal galelry of type $w$, and let $n = \ell(x)$. For $I \subset \{1, \ldots, n\}$ write $\gamma_w^I$ for the gallery obtained by folding $\gamma$ at all panels with index $i\in I$.  Among all subsets $I$ such that $\gamma_w^I$ ends at $c_y$, choose $I$ such that the sum of its elements is minimal. This ensures that $\gamma_w^I$ is $\phi_\id$-positively folded, because if $\gamma_w^I$ were not positively folded at $i \in I$, we could replace $i$ with some smaller value $j$. Specifically let $j$ be the first index such that the $i$-th and $j$-th panels of $\gamma_w^I$ lie in the same hyperplane. Then for $J$ the symmetric difference of $I$ and $\{i,j\}$
	we have $\gamma_w^J$ ending at $c_y$. Moreover $\sum(J) < \sum(I)$ because $j < i$, since every gallery from $\id$ crosses all hyperplanes from the $\phi_\id$--positive to the $\phi_\id$--negative side first. 
\end{proof}

Compare this proposition also with \cite[Lemma 2.2.1]{BjoernerBrenti}.
I'd like to emphasize that the proposition does not imply that every $\phi_+$-positively folded gallery is also $\phi_\id$-positively folded. Counterexamples can easily be constructed. 

Let us have a look at some examples of shadows. 

\begin{example}\label{ex:BruhatOrder}
	The shaded alcoves in the picture on the left of Figure~\ref{fig:BruhatOrder} are the elements of the shadow of $x$ with respect to the trivial positive orientation on a type $\tilde A_2$ Coxeter complex. By the previous proposition this is the same as the Bruhat interval $[\id, x]$ and also the same as $\Shadow_\id(x)=\Shadow_{\phi_+}(x)$.  

	The picture on the right in the same figure shows the shadow of the outlined element by an orientation at infinity. Here the chamber at infinity defining the orientation $\phi$ is indicated by the arrow pointing to a chamber at infinity of the type $\tilde A_2$ Coxeter complex. 
	For more details (also explaining the light and dark gray coloring) see \cite{Shadows}.  
\end{example}

\begin{figure}[h]
	\begin{minipage}[l]{0.45\textwidth}
		\begin{overpic}[width=0.9\textwidth]{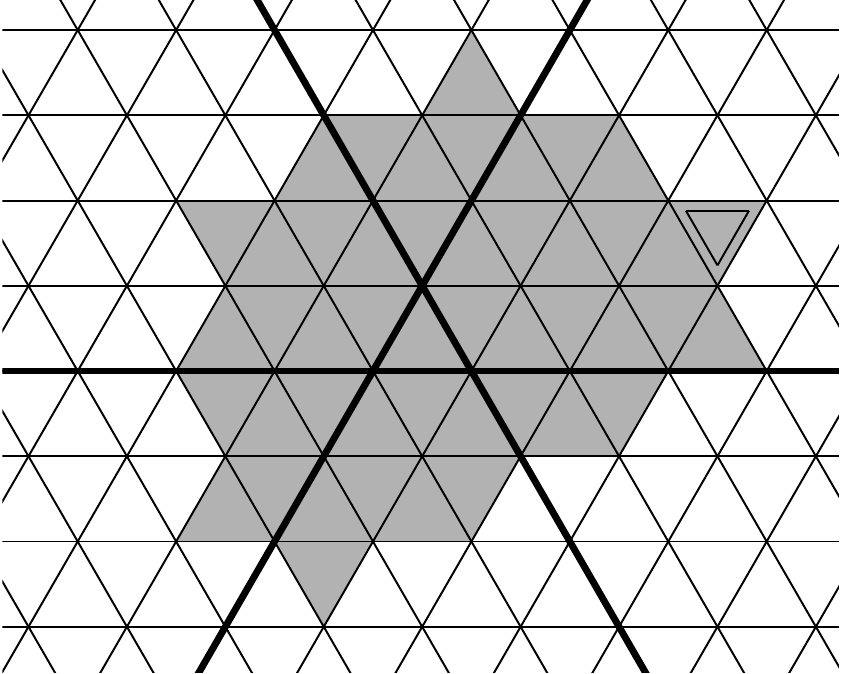}
			\put(50,37){\makebox(0,0)[cb]{$\id$}}%
			\put(85,50){\makebox(0,0)[cb]{$x$}}%
		\end{overpic}
	\end{minipage}
	\begin{minipage}[r]{0.45\textwidth}
		\includegraphics[width=0.9\textwidth, angle=180]{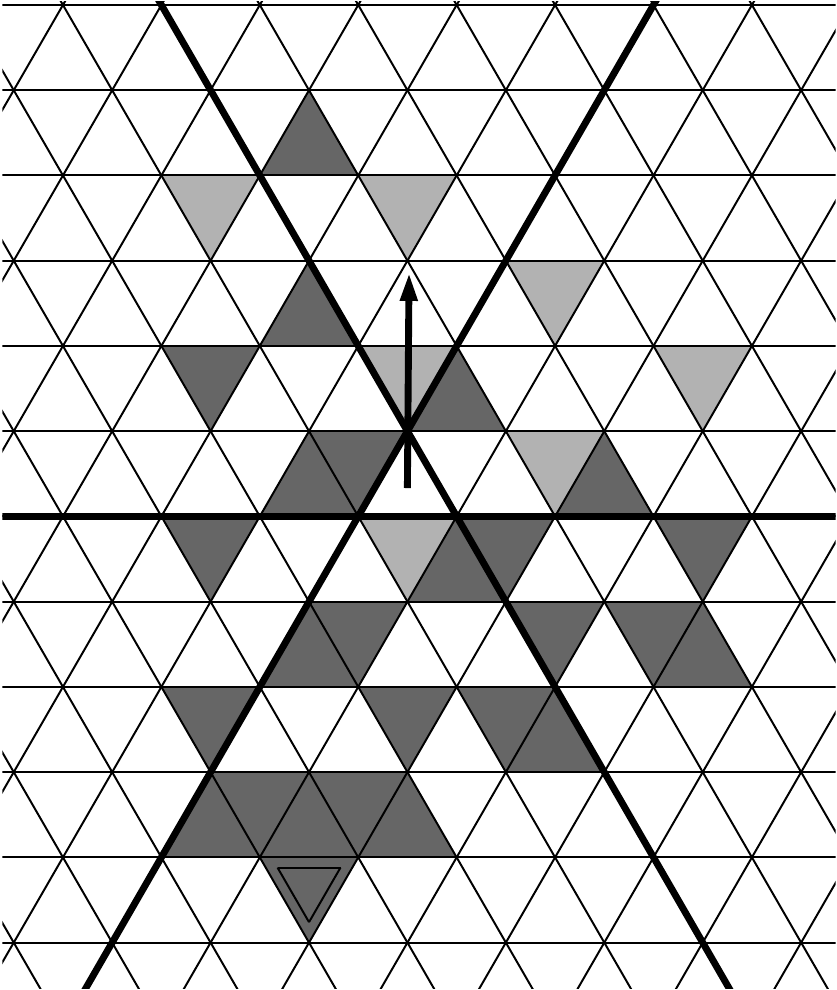}
	\end{minipage}
	\caption{The picture shows the shadow $\Shadow_{\phi_+}(x)$ (on the left) with respect to the trivial positive orientation $\phi_+$ and the shadows $\Shadow_{\phi}(x)$ of an orientation at infinity (on the right) in type $\tilde A_2$. For more details see Example~\ref{ex:BruhatOrder}. }
	\label{fig:BruhatOrder}
\end{figure}

I have mentioned earlier that it is sometimes useful to vary the notion of a combinatorial gallery a bit. The same holds true for shadows. 

\begin{remark}[Vertex-to-vertex galleries and vertex shadows]
	One natural modification of the notion of a shadows is obtained as follows. Instead of combinatorial galleries from $\fa$ to some alcove $a_x$ consider combinatorial galleries of the following shape
		\[\gamma=(\lambda_0, c_0, p_1, c_1, \ldots, p_n, c_n, \lambda_n)\]
	connecting the base vertex $\lambda_0$ with some other vertex $\lambda_n$ in the Coxeter complex. 
	Then one may look at the vertex shadows of such a gallery, respectively of the vertex $\lambda_n$. 
	The shadow of a vertex-to-vertex gallery with respect to a fixed orientation $\phi$ is the set
	$\Shadow_{\phi}^\vee$ of all vertices $\lambda \in\Sigma$ such that $\lambda$  is the end-vertex of a $\phi$-positively folded gallery of the same type as $\gamma$ that also starts in $\lambda_0$. 
	Typically $\lambda_0$ will be taken equal to the origin and $\lambda_n$ is (depending on the application) either a vertex in the co-root lattice or in the weight-lattice in $\Sigma$. 

	These vertex shadows are closely linked to the Littelmann path model and the convexity theorem we'll be talking about in Sections~\ref{sec:Littelmann} and \ref{sec:Kostant} below. 
\end{remark}

Maybe the most natural question at hand is to decide which $y$ are in $\Shadow_{\phi}(x)$ for a given $\phi$ and $x\in W$. Unfortunately there are no closed formulas available for these sets, yet. 
The best we can do is provide recursive descriptions of shadows similar to the recursive descriptions of Bruhat order. Here is a first example for such a recursive result on shadows. Compare \cite{Shadows} or \cite{MNST} for more results and proofs.

Note that the condition $\v_\varphi(s) < 0$ (resp. $>0$) in the next theorem simply means that the alcove corresponding to $s$ in $\Sigma$ is on the negative (resp. positive)  side of the hyperplane separating $s$ from $\id$. When one replaces the Weyl chamber orientation by the trivial positive orientation this result reduces to a recursive description of Bruhat order intervals which can be proven similarly. 

\begin{thm}[Recursive computation of shadows, {\cite[Thm.7.1]{Shadows}}]
	\label{thm:regular_shadow}
	Let $\varphi$ be a Weyl chamber orientation on $\Sigma=(\aW, S)$. Then for all $x \in \aW$ and $s \in S$ the following holds.
	\begin{enumerate}[label=(\roman*)]
		\item\label{item:DR} If $s$ is such that $l(xs)<l(x)$, then 
		\[
		\Shadow_\varphi(x) = \Shadow_\varphi(xs) \cdot s \cup \{z \in \Shadow_\varphi(xs): \v_\varphi(zs) < \v_\varphi(z)\}.
		\]
		\item\label{item:DL} If $s$ is such that $l(sx)<l(x)$, then
		\[
		\Shadow_\varphi(x) = \left\{\begin{array}{ll} 
			s \cdot \Shadow_{s\varphi}(sx) \cup \Shadow_\varphi(sx) & \text{ if } \v_\varphi(s) < 0  \\	
			s \cdot \Shadow_{s\varphi}(sx) & \text{ if } \v_\varphi(s) > 0.  
		\end{array} \right.
		\]
	\end{enumerate}
\end{thm}


For more results on the recursive behavior of shadows with respect to Weyl chamber orientations in affine Coxeter groups compare Algorithms R and L in Section 7.2 as well as Theorem 7.2 in \cite{Shadows}. More general orientations with respect to chimneys were introduced by Mili\'cevi\'c, Naqvi, Thomas and myself in \cite{MNST}. See in particular Theorem 1.1 of \cite{MNST} where the recursive behavior of a shadow with respect to a chimney orientation is described. An example of such a chimney is provided in Figure~\ref{fig:AlcoveShadowA2}. 
How shadows of chimneys relate to sets of positively folded galleries to certain orbits in (partial) affine flag varieties is explained in Theorems 1.2 and 1.3 of \cite{MNST}. These theorems are flag variety versions of results in \cite{GaussentLittelmann} and \cite{PRS} in the affine Grassmannian which we will get to know in the next section.

\begin{figure}[ht]
	\begin{center}
		\resizebox{0.5\textwidth}{!}
		{
			\begin{overpic}{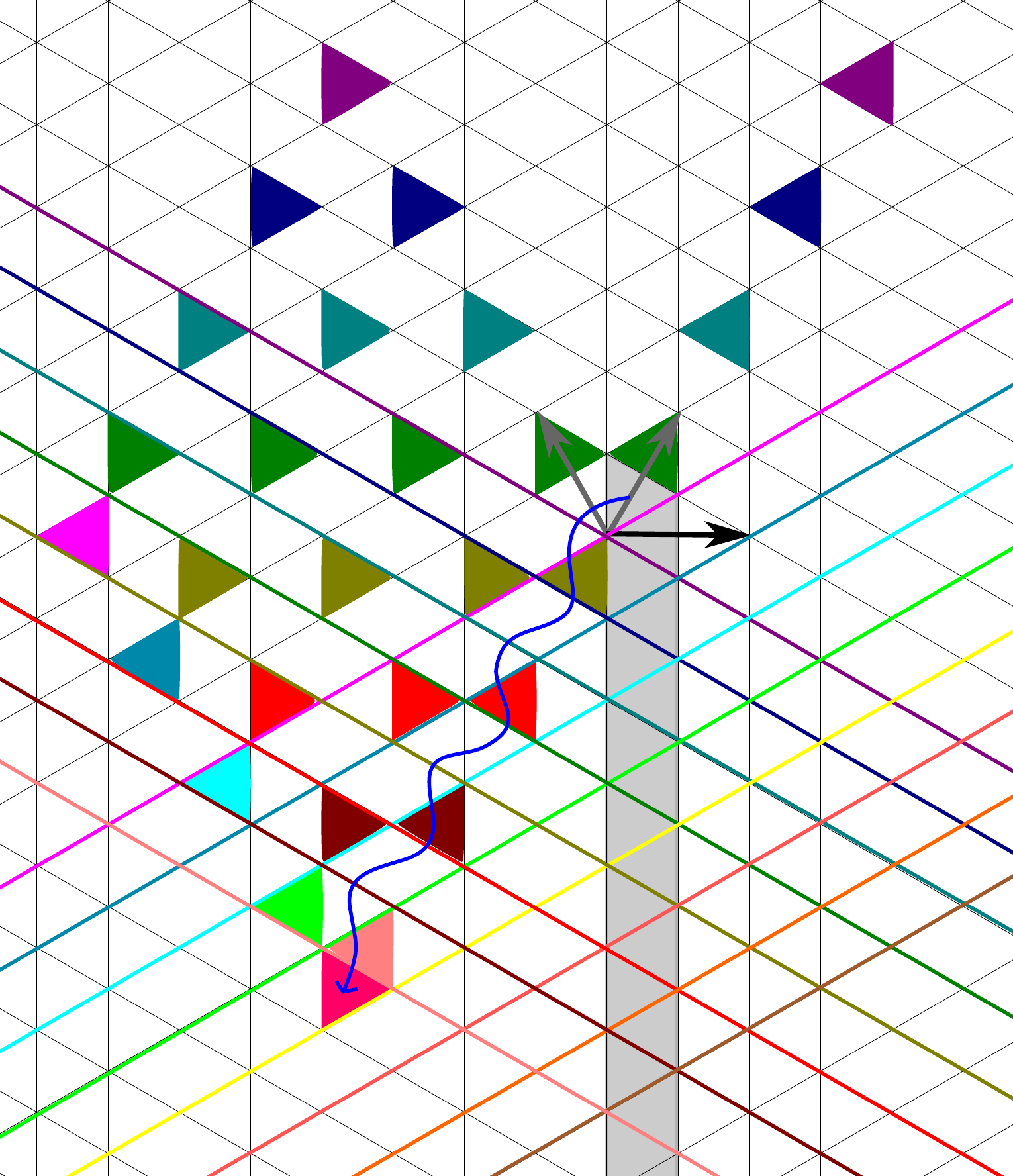}
				\put(58.5,55.5){$\alpha$}
				\put(28,14){$x$}
			\end{overpic}
		}
		\caption{The colored-in alcoves are the shadow of the pink alcove $x$ with respect to the chimney represented by the gray shaded region. See \cite{MNST} for further details.}
		\label{fig:AlcoveShadowA2}
	\end{center}
\end{figure}


\section{Back to the roots}
\label{sec:gallerymodel}

To the best of my knowledge the first major application of positively folded galleries was by Gaussent and Littelmann in \cite{GaussentLittelmann} where they are used to describe a dense subset of an MV-cycle. The folded galleries are the combinatorial analogs of folded paths, the main tool in Littelmann's path model \cite{Littelmann1, Littelmann2}. 
Folded or stammering, respectively stuttering, galleries (without the condition of positivity with respect to an orientation) appear in many places in the literature. The positively folded galleries appeared under the name of distinguished subexpressions in the work of Deodhar \cite{Deodhar}. The term folded galleries was first used in the PhD thesis of Gaussent. 

In this section I would like to highlight and explain parts of the results of the gallery model introduced in \cite{GaussentLittelmann} which interprets the path model in terms of the geometry of the associated affine Grassmannian. We will also see how this is related to the theory of buildings and to the notion of shadows introduced above. 

The material in this section is also relevant for the applications reported on in \Cref{sec:Kostant,sec:ADLV}. 
 
Littelmann answered some of the most basic questions in the representation theory of symmetrizable Kac-Moody algebras in a uniform and elegant way. Using folded paths and root operators on the collection of all paths he was able to  
\begin{itemize}
	\item compute the weight multiplicities of irreducible highest weight representations, and
	\item determine the irreducible components of the tensor product of two highest weight representations.  
\end{itemize} 

The path model provides a uniform answer to these (and more) questions by explicit combinatorial formulas counting certain piece-wise linear positively folded paths. These questions were also answered by the notion of crystal structures due to Kashiwara \cite{Kashiwara} and Lusztig \cite{LusztigCanonical}. Littelmann's approach generalizes the standard monomial theory of Lakshmibai and Seshadri \cite{LakshmibaiSeshadri} and is as well an instance of a crystal structure.

\subsection{Positively folded galleries of maximal dimension} 

In this section the notion of a dimension of a positively folded gallery is introduced. The dimension counts, loosely said, how many positive folds and crossings there are in a given gallery.
 
We are only considering vertex-to-vertex galleries in affine Coxeter complexes $\Sigma$ in this section and will not explicitly mention this standing assumption in all places. 
Such a vertex-to-vertex gallery is called \emph{minimal} if it has shortest possible length among all vertex-to-vertex galleries connecting the same pair of endpoints.

\begin{remark}
	Our notion of combinatorial galleries is not exactly the same as the definition of that term provided in \cite{GaussentLittelmann}. In \cite{MST} we discuss the difference between these notions and also discuss the relationship between minimal vertex to vertex galleries and minimal alcove to alcove galleries which are (unfortunately) not exactly equivalent. 
\end{remark} 

Let $\phi$ be a Weyl chamber orientation on $\Sigma$. To every gallery we may associate a dimension by counting a specific class of  hyperplanes, namely the load-bearing ones. Following \cite{GaussentLittelmann} we define those for a fixed vertex-to-vertex gallery  $\gamma=(\lambda_0, c_0, p_1, c_1, p_2, \ldots p_n, c_n, \lambda_n)$ as follows. 
	
\begin{definition}[load-bearing hyperplanes]\label{def:load-bearing}
 A hyperplane $H$ in a Coxeter complex is \emph{load-bearing} for a gallery $\gamma$ at $p_i$ (or $\lambda_0$) if $H$ contains $p_i$ (or $\lambda_0$) and  $c_i$ (respectively $c_0$) is on the positive side of $H$. 
\end{definition}

A hyperplane $H$ is load-bearing if $H$ separates $c_i$ from the chamber at infinity defining $\phi$. For the panels $p_i$ a hyperplane is load-bearing if the gallery $\gamma$ either has a positive fold at $p_i$ or if the gallery crosses $H$ from the negative side (where $c_{i-1}$ lies) to the positive side (where $c_i$ lies). At every panel there is hence at most one load-bearing hyperplane but the same hyperplane may appear to be load-bearing at more than one panel. At the vertex $\lambda_0$ there can be more than one load-bearing hyperplane.  

Note that \cite{GaussentLittelmann} and many other references refer to hyperplanes as walls and hence call the objects defined in \ref{def:load-bearing} load-bearing walls instead. 

\begin{definition}(Gallery dimension)\label{def:dimension}
	The \emph{dimension of a gallery} is the number of its pairs $(p_i, H)$ with $H$ a  load-bearing hyperplane at $p_i$.  
\end{definition}

Let's look at an example. Suppose the gallery $\gamma$ starts at the origin $\lambda_0$ and has as its first alcove the fundamental alcove $\fa$ corresponding to the identity. Suppose in addition that the Weyl chamber orientation $\phi$ is induced by the opposite of the fundamental Weyl chamber. Then every hyperplane containing the origin is load-bearing for $\lambda_0$ and hence the number of hyperplanes load-bearing for $\lambda_0$ is the same as the number of reflections in the spherical Weyl group or the number of positive roots in the underlying root system. For all other choices of Weyl chamber orientations the number of load-bearing hyperplanes is bounded above by that amount.

Dimensions behave well with respect to translations of galleries and also other natural mutations many of which are illustrated and explained in detail in Section 4.1 of \cite{MST}. 

It was shown in \cite[Prop. 3]{GaussentLittelmann} that the dimension of galleries of a fixed type, of which there may be many, is bounded above by a certain number that solely depends on the type and the end-vertex. The ones that attain this upper bound play a very specific role and hence deserve their own name.

\begin{definition}[Lakshmibai-Seshadri galleries]\label{def:LSgalleries}
	A gallery $\gamma$ of a fixed type $\tau$ and fixed end-vertex that has the maximal possible dimension is called \emph{LS-gallery} of type $\tau$. 
\end{definition} 

One can obtain all  LS-galleries of a fixed type (ranging over all possible endpoints) by repeated applications  of certain root operators $e_\alpha$ and $f_\alpha$ to one given $LS$-gallery. These operators exists for any choice $\alpha$ of a parallel class of hyperplanes in $\Sigma$, called a root. 
These operators are tools to manipulate positively folded and in particular LS-galleries in a very controlled way that allows to keep track of how their dimension changes along the way. 

The main idea is the following:   
Fix a Weyl chamber orientation in $\Sigma$ and a vertex-to-vertex gallery $\gamma$ starting at the origin $\lambda_0$. These operators pull, respectively push positive folds from one hyperplane in a given parallel class to the next hyperplane. If $\gamma$ has no fold along a hyperplane in the parallel class $\alpha$, then one of the two operators may be applied to introduce a new fold. I should mention that not every operator is defined for every gallery. However, for every gallery there is always some $\alpha$ such that either $e_\alpha$ or $f_\alpha$ is defined.

\begin{figure}[htb]
	\begin{center}
		\begin{overpic}[width=0.5\textwidth]{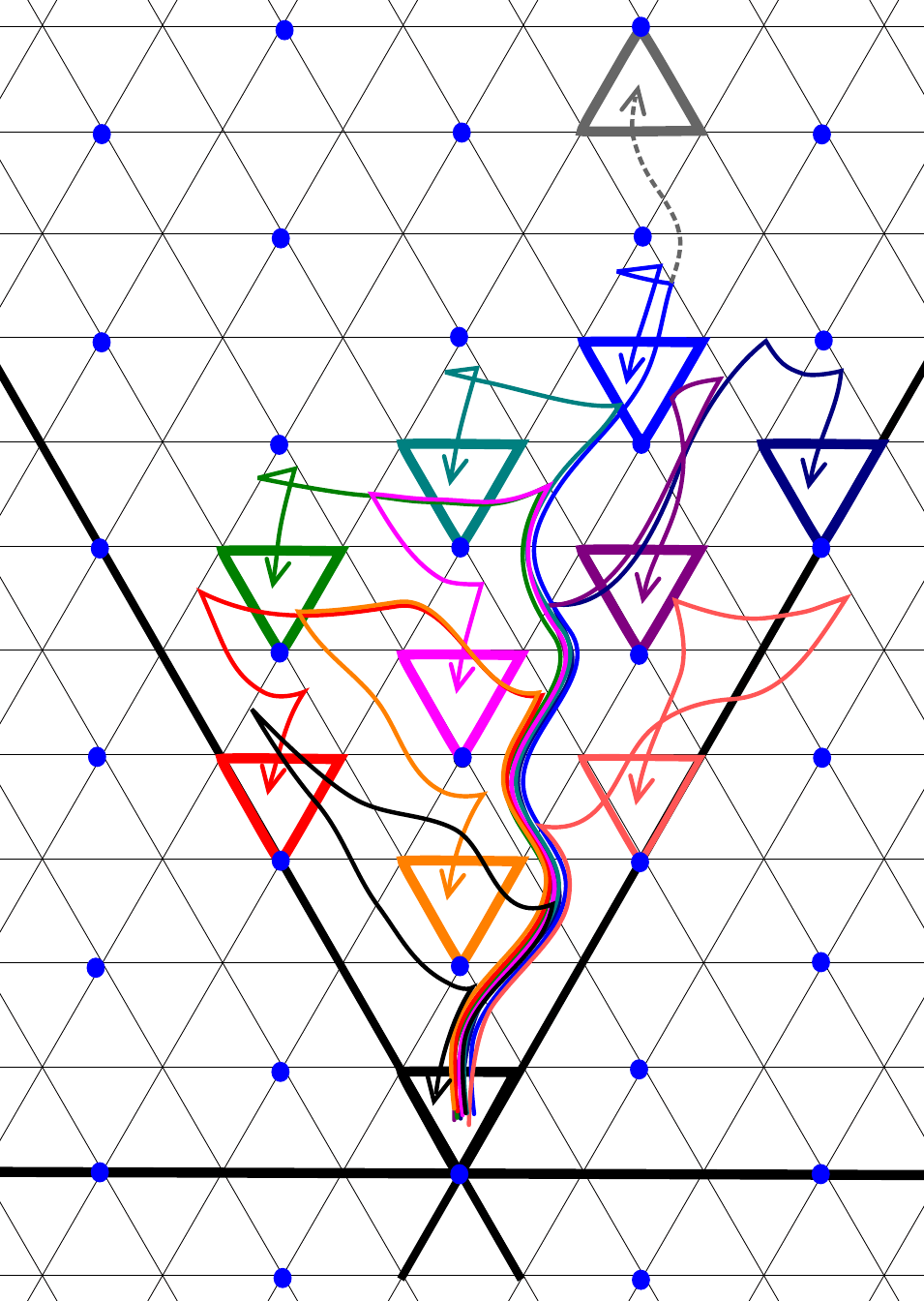}
			\put(0,60){{$H_{1}$}}
			\put(72,70){{$H_{2}$}}
		\end{overpic}	
	\end{center}
	\caption{Applying available root operators to explicitly constructed galleries. This implies that ADLVs $X_x(b) \neq \emptyset$ for most $b=t^\mu$ between 1 and $x$. See Section~\ref{sec:ADLV} for details. The figure is taken from \cite{MST}.}
	\label{fig:root-operators}
\end{figure}

Some applications of root operators are illustrated in Figure~\ref{fig:root-operators}. For example  the orange gallery is obtained from the red by an application of the root operator associated with the class of hyperplanes parallel to the diagonal labeled $H_2$. The green galleries are obtained from the dark blue gallery by one, respectively two applications of a root operator corresponding to the class of Hyperplanes parallel to the diagonal labeled $H_1$.  

The main properties of root operators are as follows. Compare Cor. 2 and Lemmas 5,6,7 and Prop. 5 in \cite{GaussentLittelmann} for a proof.  

\begin{prop}[Properties of root operators]
Suppose $\gamma$ is an LS-gallery in some Coxeter complex $\Sigma$. Then the following hold. 	
\begin{enumerate}
	\item An application of $e_\alpha$ increases the dimension of $\gamma$ by one. 
	\item An application of $f_\alpha$ decreases the dimension of $\gamma$ by one.
	\item If $e_\alpha$ is defined for $\gamma$, then $f_\alpha$ is defined for $e_\alpha(\gamma)$ and $f_\alpha(e_\alpha(\gamma))=\gamma$ and similarly with roles of $e_\alpha$ and $f_\alpha$ switched. 
	\item The set of LS-galleries is stable under applications of $e$ and $f$ operators. 
	\item Any LS gallery of type $\tau(\gamma)$ starting at the same point as $\gamma$ can be obtained from $\gamma$ by a finite number of applications of $e$ or $f$ operators. 
\end{enumerate}
\end{prop}

These properties already hint at the close connection between LS-galleries, root operators and crystal structures in representation theory. I will now explain how to turn the collection of LS-galleries into a graph and then how this is related to crystal graphs of irreducible representations of $G^\vee$ of some fixed highest weight. 

Suppose $G$ is a connected complex semisimple algebraic group. Let $T$ be a fixed maximal torus in $G$ and let $X=X(T)$ be its character group. Then the group of co-characters $X^\vee=\Mor(\C^\ast, T)$ is the character group of the dual maximal torus $T^\vee$ in $G^\vee$. The dominant co-characters $\lambda\in X^\vee_+$ classify the irreducible, finite-dimensional $G^\vee$ modules $V(\lambda)$ while appearing as their highest weights. 
These elements $\lambda\in X^\vee_+$ may also be viewed as vertices in $\Sigma$ that are contained in the fundamental Weyl chamber.  

Fix a minimal vertex-to-vertex gallery $\gamma$ connecting the origin with a weight $\lambda$ and denote by $\Gamma_{LS}(\gamma)$ the collection of all LS-galleries of the same type as $\gamma$. Define  $B(\gamma)$ to be the directed, edge colored graph whose vertices are equal to the set of all LS-galleries of type $\tau(\gamma)$. There exists a directed edge from vertex $\delta_1$ to $\delta_2$ in $B(\gamma)$ whenever $f_\alpha(\delta_1)=\delta_2$ for some root $\alpha$.

\begin{thm}[Crystals and characters {\cite[Thm A]{GaussentLittelmann}}]\label{thm:character}
	 Let $\gamma$, $\lambda$ be as above. Then the graph $B(\gamma)$ is connected and isomorphic to the crystal graph of the irreducible representation $V(\lambda)$ of $G^\vee$ of highest weight $\lambda$. Then
	 \[\mathrm{Char} V(\gamma)=\sum_{\delta\in\Gamma_LS(\gamma)} exp(end(\delta))\] 
	 where $end(\delta)$ denotes the end-vertex of $\delta$. 
\end{thm}

This theorem establishes a connection between purely combinatorial objects in the Coxeter complex, namely the LS-galleries, and highest weight representations and their characters. 

One consequence of this result is the following.

\begin{prop}
	Let $\lambda$ be a dominant vertex in $\Sigma$. Write $\phi$ for the Weyl chamber orientation  induced by the anti-dominant Weyl chamber. 
	Then \[\Shadow_{\phi}^\vee(\lambda)=\{\nu\leq \lambda\mvert \nu \text{ a vertex in } \Sigma \text{ with }\tau(\nu)=\tau(\lambda)\}.\]  
\end{prop}

This relation between LS-galleries and shadows is one of the key ingredients in the proof of the Convexity Theorem~\ref{thm:KostantConvexity}. 
Before we talk about  this result we need to learn more about buildings.


\subsection{Galleries, buildings and flag varieties}\label{sec:Littelmann}

Folded galleries and even root operators have very natural interpretations in the context of (Bruhat-Tits) buildings. We will now explain what a building is and how they give rise to folded galleries. 
All of this is also closely related to semisimple algebraic or, more generally, reductive groups and their affine Grassmannians.

\begin{definition}[Buildings]\label{def:buildings}
	A \emph{building} is a simplicial complex $X$ which is the union of a collection $\App$ of subcomplexes called \emph{apartments} such that the following holds: 
	\begin{enumerate}
		\item Each $A\in\App$ is isomorphic to a Coxeter complex $\Sigma=\Sigma(W,S)$.
		\item Any two simplices in $X$ lie in common apartment. 
		\item If for two simplices $c,c'$ in $X$ there are two apartments $A,A'$ containing both $c$ and $c'$ then there exists an isomorphism from $A$ to $A'$ which pointwise fixes $A\cap A'$.  
	\end{enumerate}
	The set $\App$ is called \emph{atlas} of $X$. 
\end{definition}

From the definition of a building one can directly deduce the following lemma. 

\begin{lemma}
All apartments in an atlas $\App$ of a building $X$ are isomorphic to the same Coxeter complex $\Sigma$.
\end{lemma} 
\begin{proof}
	 Suppose there are apartments $A$ and $A'$ isomorphic to $\Sigma$ and $\Sigma'$ respectively. Then for any choice of simplex $c\in A$ and $c'\in A'$ there exists an apartment $A''$ containing both $c$ and $c'$. The apartment $A''$ is, by item (1), also isomorphic to some Coxeter complex $\Sigma''$. Now apply item (3) to pairs of simplices lying in $A\cap A''$, and also in $A'\cap A''$ to see that $\Sigma$ and $\Sigma'$ are indeed both isomorphic to $\Sigma''$. Hence the claim. 	
\end{proof}

We may thus refer to the type of the Coxeter system related to the apartments in a building $X$ as the \emph{type of the building $X$}. 
In the following we will restrict to \emph{affine} buildings, that is those whose apartments are affine Coxeter complexes. 

By definition an affine building is the union of a collection of apartments all of which are isomorphic to some tiled $\R^n$.  
Again by slight abuse of notation and convention we will not distinguish between an affine building and its geometric realization. 

For more details on (affine) buildings see one of the many excellent books on this topic. For example Brown's classic introduction \cite{Brown}, the monograph by Abramenko and Brown \cite{AbramenkoBrown} or Garret's and Ronan's books \cite{Garrett, Ronan}. How buildings fit into the theme of geometric group theory and CAT(0) spaces is explained in \cite{BH} and \cite{Davis}.

\begin{example}[Affine buildings of dimension one]
	A first example of an affine building is easily given in dimension one. In this case there is only one possible complex $\Sigma$ playing the role of the apartments: the tiling of the number line by unit intervals which is the Coxeter complex $\Sigma$ of the infinite dihedral group. 
	Now any simplicial tree without leaves is an example of an affine building of type $D_\infty$. It is not hard to check that such a tree is the union of infinitely many copies of $\Sigma$. 
	One such building is shown in Figure~\ref{fig:tree_p2}. 
\end{example}

\begin{figure}[htb]
	\begin{center}
		\begin{overpic}[width=0.9\textwidth]{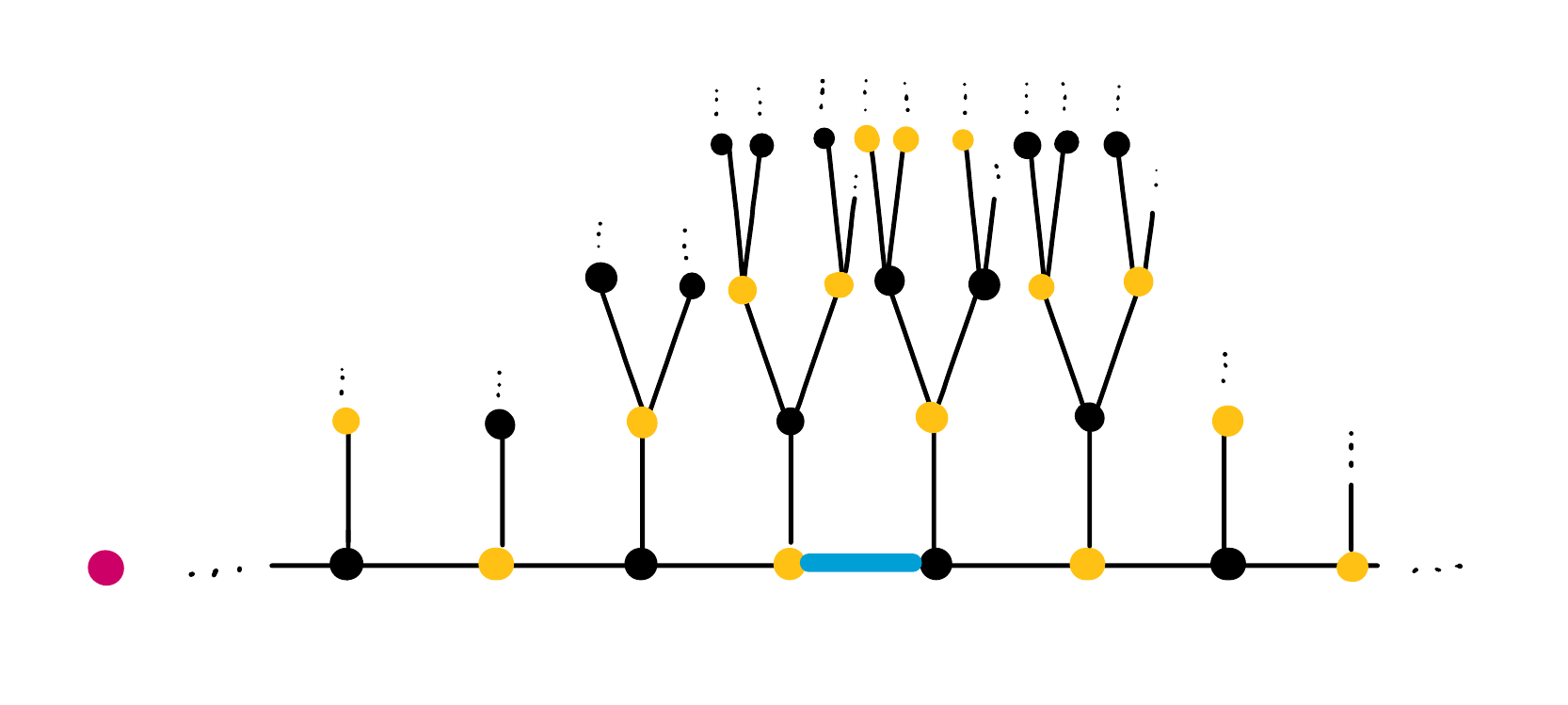}
			\put(4, 3){{$\infty$}}
			\put(54,11){{$c_0$}}
			\put(58, 3){{$\lambda_0$}}
		\end{overpic}	
		\caption{The affine Bruhat-Tits building for $SL_2(\mathbb{Q}_2)$ with respect to the $2$-adic valuation.}
		\label{fig:tree_p2}
	\end{center}
\end{figure} 

Dimension one is essentially the only dimension where interesting examples of affine buildings can be given in such an ad-hoc and explicit way. In higher dimensions examples are constructed by means of an algebraic groups. We call such buildings \emph{of algebraic origin}. From dimension three on all buildings are of algebraic origin. This is a consequence of their classification outlined by Tits in \cite{TitsComo}. A detailed proof of the classification of affine buildings can be found in \cite{WeissAffine}. Dimension two also allows for free constructions of buildings not of algebraic origin. Such constructions appeared for the first time in \cite{Ronan-free1} and \cite{RonanTits-free2} followed by more examples and constructions since then. We won't go into the details here but refer the reader to the literature instead. See for example \cite{BerensteinKapovich, FunkStrambach, RonanConstruction, DavisConstruction} and the references therein.

Recall that at infinity of an affine Coxeter complex there is a spherical simplicial complex on which the associated spherical Weyl group acts. One can prove that at infinity of an affine building exists a spherical building. This building at infinity is the union (or gluing) of all the spheres at infinity of its apartments. So in particular to any Weyl chamber in an apartment of an affine  building $X$ corresponds a chamber, i.e.  maximal simplex, in the  simplicial building at infinity. 

Suppose we have a reductive group with a BN-pair. For example $SL_n$ over $F=\C((t))$ or some other field with a discrete valuation $\nu$. To such a group one may associate an affine Bruhat Tits building $X=X(G, F, \nu)$ and a spherical building $\partial X=\partial X(G,F)$ at infinity of $X$. The apartments of $X$ and $\partial X$ are in one-to-one correspondence and chambers in $\partial X$ correspond to parallel classes of Weyl chambers in $X$. 
The group $G$ acts transitively on the apartments and alcoves of its Bruhat-Tits building and also on the pairs of alcoves and apartments containing them. A maximal torus $T$ acts as the translation group of the affine Weyl group on a fixed base apartment $A_0$ in $X$. 

In the tree case the situation is as follows. 

\begin{example}
	By Serre's construction \cite{SerreTrees} there is a tree $T$  without leaves associated with $SL_2$ over a non-archimedian local field $F$ with discrete valuation. The number of edges at any vertex is determined by the order of the residue field of the valuation. In case $F=\Q_p$ every vertex is contained in $p+1$ edges. In case $F=\C((t))$ vertices are contained in infinitely many edges. In particular trees associated with $SL_2$ are regular, that is the branching is the same at every vertex. 
	
	Figure~\ref{fig:tree_p2} shows this tree for $p=2$. The base apartment $A_0$ is the simplicial line that is drawn horizontally. The origin is labeled $0$, the fundamental alcove highlighted in blue and the anti-dominant Weyl chamber in $A_0$ determines the boundary point labeled $\infty$. 
\end{example}

Let's go back to the groups for a bit. 
Associated with $G$ is also the affine Grassmannian which, as a set, is defined as the quotient $\mathcal{G}=G(F)/G(\mathcal{O})$ with $\mathcal{O}$ being the valuation ring. So if for example $F=\C((t))$ then $\mathcal{O}=\C[[t]]$ the ring of Laurent polynomials in $t$. 

The interested reader may have a look at for example Zhu's lecture notes \cite{Zhu} for an introduction to affine Grassmannians. 

Since $G(\mathcal{O})$ may be identified with the origin in $A_0$, the points in the affine Grassmannian may be interpreted as vertices in the Bruhat-Tits building associated with $G$. 
As explained above the group $T$ acts as translations on a fixed base apartment $A_0$ in $X$.

Associated to every affine building $X$ are two kinds of retractions. The first is defined by a choice of a maximal simplex in an apartment.  

\begin{definition}[Retraction based at an alcove]\label{def:alcoveretractions}
	Fix an apartment $A$ in a building $X$ and an alcove $c\in A$. The \emph{retraction $r_{A,c}:X\to A$} from $X$ to $A$ \emph{based at $c$} is the map defined as follows: 
	
	For any alcove $d$ in $X$ choose an apartment $A'$ containing $c$ and $d$. (Such an $A'$ exists by the second buildings axiom.) Then define $r_{A,c}(d)$ to be the image of $d$ under the isomorphism from $A'$ to $A$ (provided by the third buildings axiom).   
\end{definition} 

Retractions based at alcoves are defined in buildings of any kind. The second kind of retraction is defined by a choice of direction at infinity in an affine building. 

\begin{definition}[Retraction based at infinity]\label{def:infinityretraction}
	Fix an apartment $A$ in an affine building $X$ and an chamber $C\in \partial A$, the boundary of the apartment $A$ in $\partial X$. The \emph{retraction $\rho_{A,C}:X\to A$} from $X$ to $A$ \emph{based at $C$} is the map defined as follows: \newline
	For any alcove $d$ in $X$ choose an apartment $A'$ containing the alcove $d$ and a Weyl chamber in $A$ representing $C$. Then define $\rho_{A,C}(d)$ to be the image of $d$ under the isomorphism from $A'$ to $A$ (provided by the third buildings axiom).   
\end{definition} 

See \cite{Brown} for a proof that the choice of $A'$ in the definition of the retraction at infinity is indeed always possible and for a proof that both retractions are well-defined. 

The following example illustrates these retractions in an affine building of dimension one.

\begin{figure}[htb]
	\begin{center}
		\includegraphics[width=0.33\textwidth]{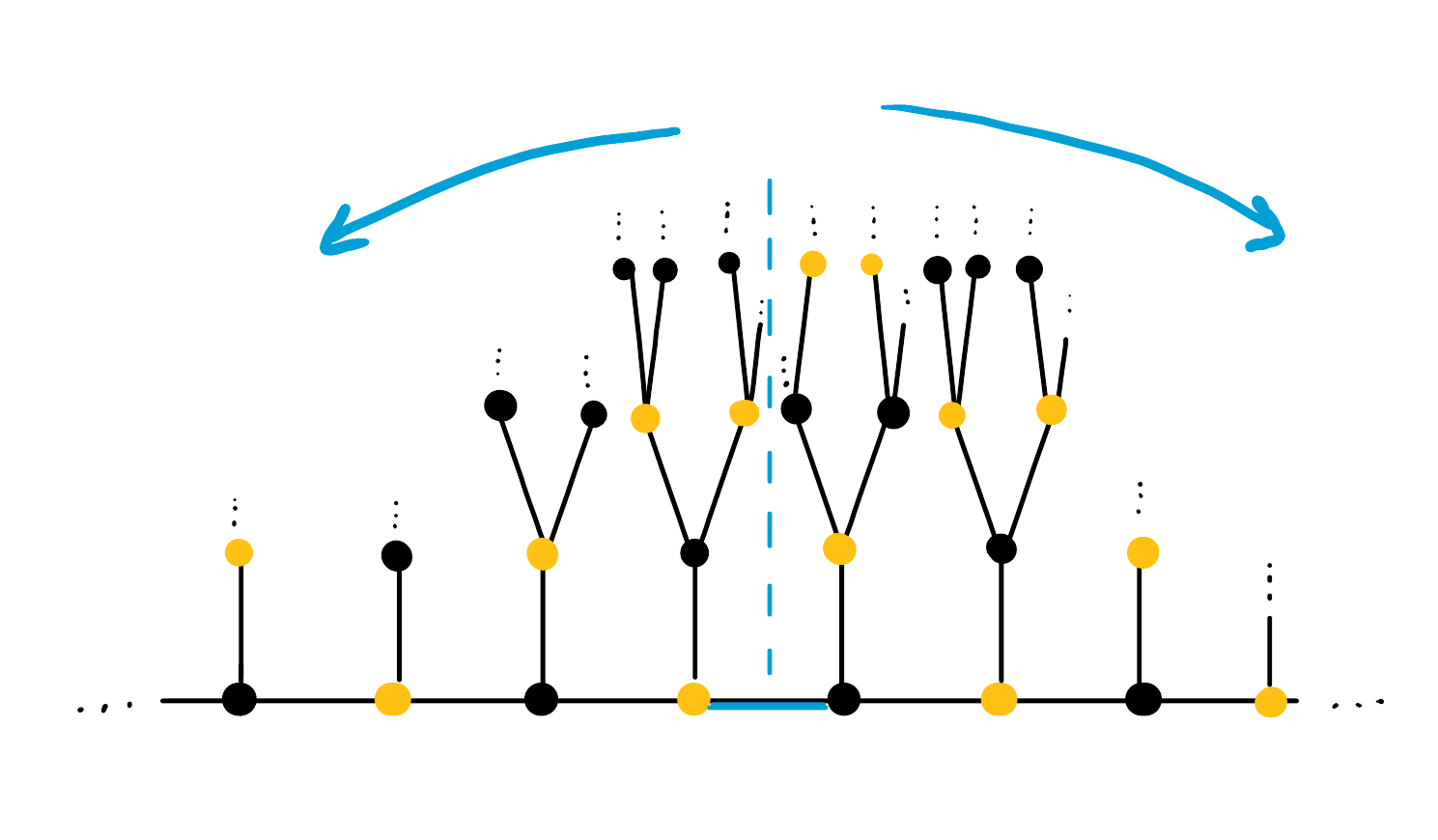}
		\includegraphics[width=0.36\textwidth]{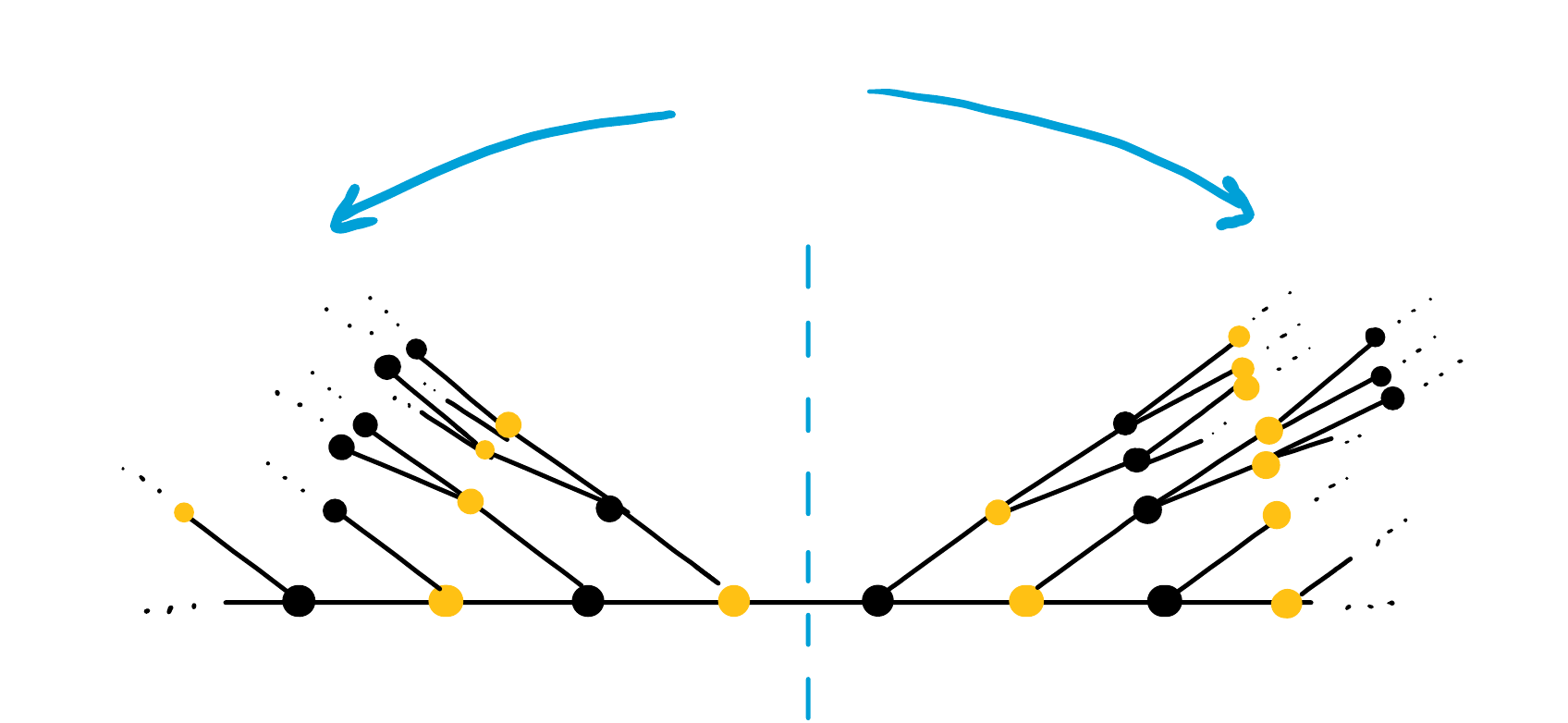}
		\includegraphics[width=0.3\textwidth]{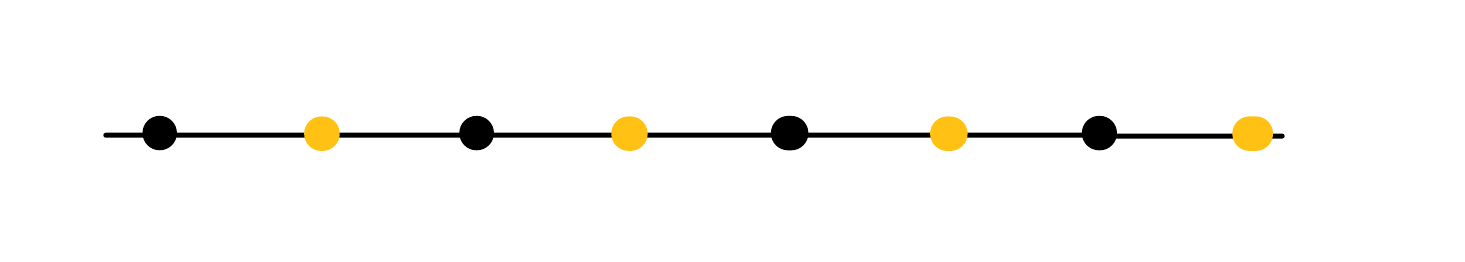}
	\end{center}
	\caption{The retraction based at the fundamental alcove flattens the tree outwards. }
	\label{fig:retraction1}
\end{figure}

\begin{example}[Trees again. And retractions.]
We return to the tree $T$ associated with $SL_2(\Q_p)$. The goal here is to explain what the retractions based at an alcove, respectively at infinity, look like in this tree case. 

First look at the retraction $r:X\to A_0$ based at the fundamental alcove shown in blue in Figure~\ref{fig:retraction1}. The blue alcove is the same as the one labeled $c_0$ in Figure~\ref{fig:tree_p2}. This retraction flattens the whole tree onto $A_0$ outwards and away from the blue alcove. Every bi-infinite line in the tree that contains the blue edge has one side that sticks out to the left (and up) and one side sticking out to the right (and up) of that blue alcove. Any such line is laid flat on top of the horizontal line $A_0$. 
For all $t\in T$ one has $r^{-1}({W_0}.t) = K.t$ and $r^{-1}(t)=I.t$, where $K\define\Stab(\lambda_0)$ is the stabilizer of the origin and  $I\define \Stab(\mathbb{1})$ stabilizes the fundamental alcove in $A_0$.
	
A Weyl chamber in the tree is just a ray and hence the retraction based at a class at infinity is defined by a boundary point of $T$. Take as point at infinity the direction represented by the pink dot which was labeled $\infty$ in Figure~\ref{fig:tree_p2}. The corresponding retraction flattens the entire tree onto $A_0$ away from the chosen point at infinity. For the boundary point labeled $\infty$ representing the anti-dominant direction in $A_0$ the effect of the retraction is illustrated in Figure~\ref{fig:retraction_infty}.  
For all $t\in T$ one has $\rho_{\infty}^{-1}(t)=U.t$, where $U=U(F)=Stab_G(\infty)$ is the subgroup of $G$ which stabilizes the direction labeled $\infty$. 
\end{example}

\begin{figure}[htb]
	\begin{center}
		\scalebox{-1}[1]{\includegraphics[width=0.33\textwidth]{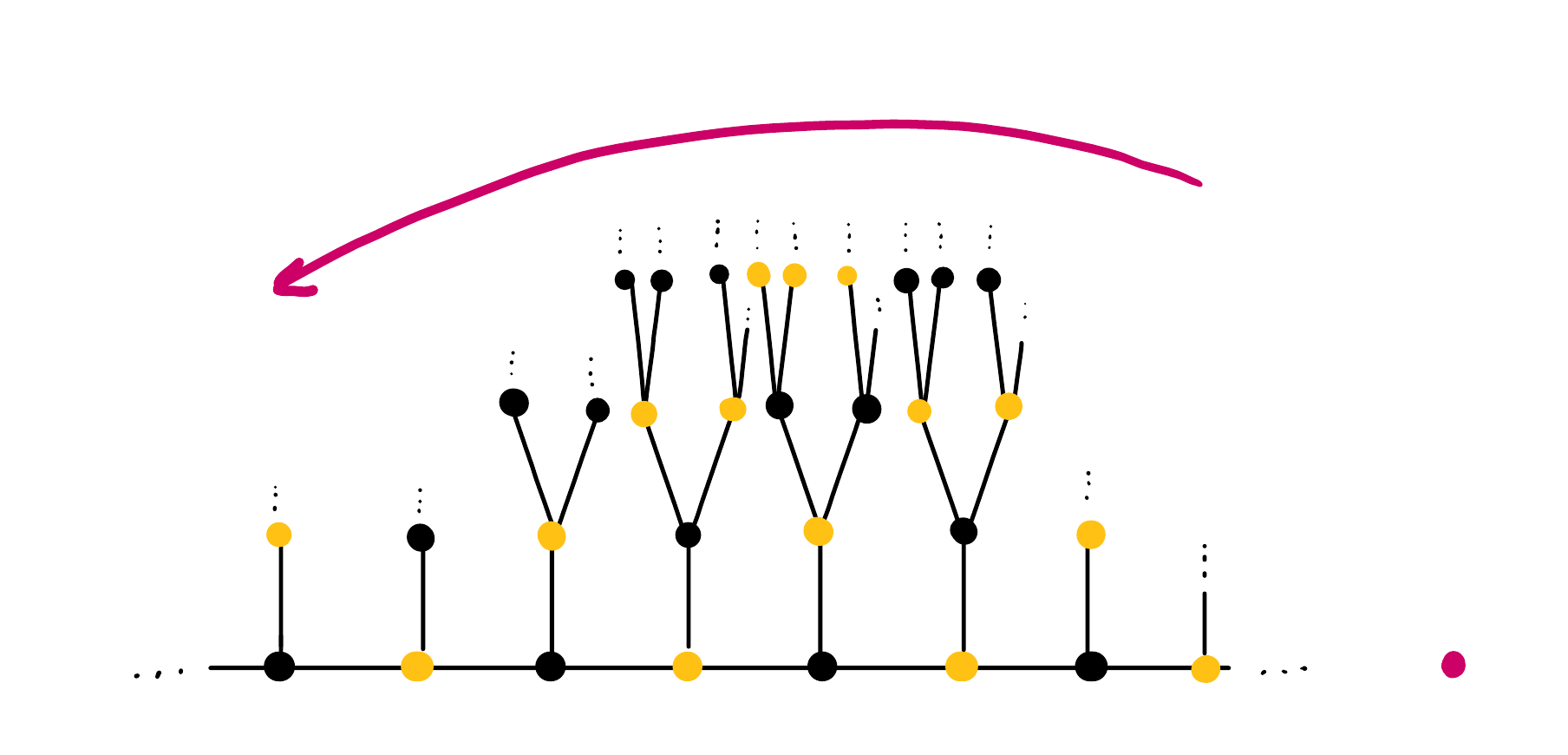}}
		\scalebox{-1}[1]{\includegraphics[width=0.36\textwidth]{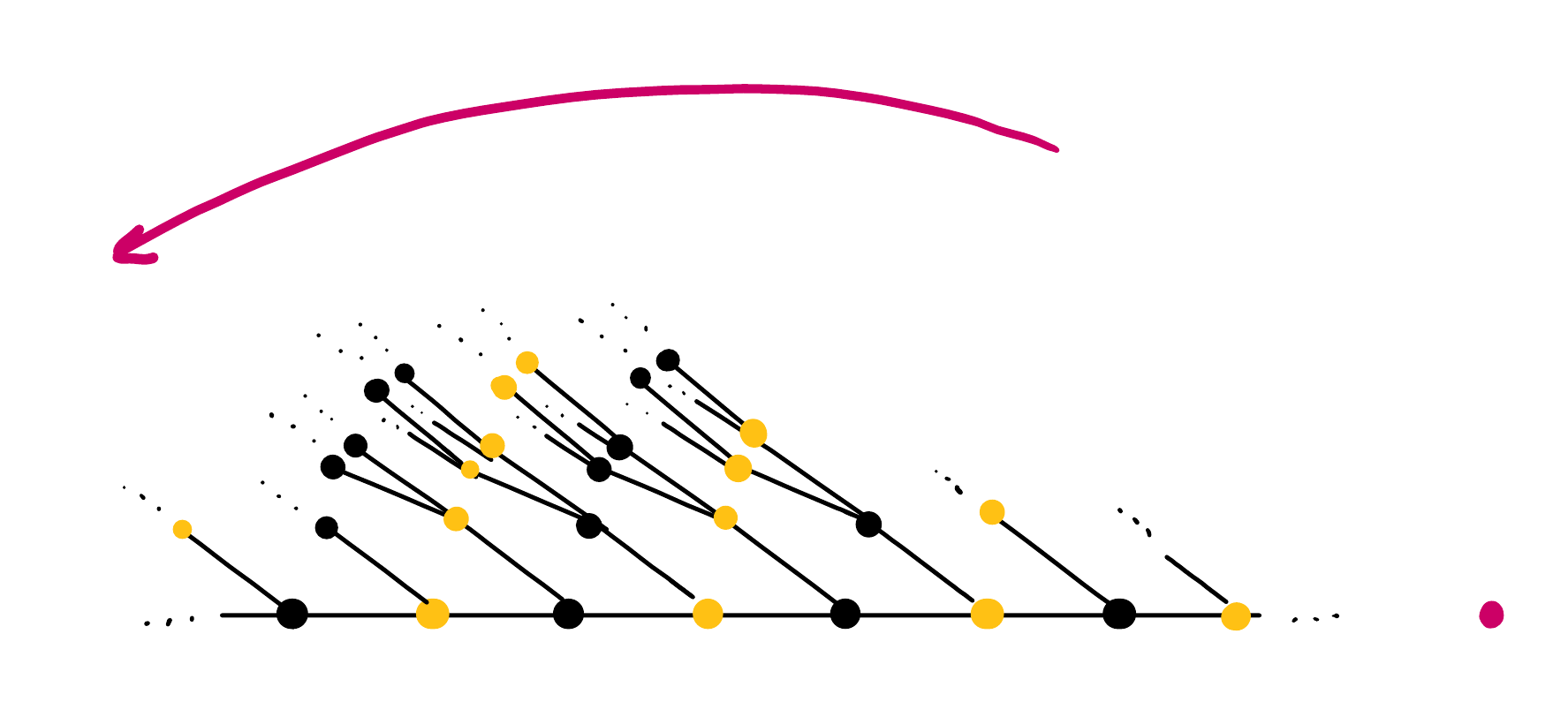}}
		\includegraphics[width=0.3\textwidth]{A1}
	\end{center}
	\caption{The retraction of a tree based at a Weyl chamber at infinity, indicated by the pink dot to the left of the horizontal line, flattens the tree away from that direction at infinity. }
	\label{fig:retraction_infty}
\end{figure}

These retractions are closely related to orbits of subgroups of $G$ as well as to (folded) galleries. Call $U$ the unipotent radical of the Borel subgroup of $G$. Then one has the following. This group is the stabilizer of the opposite fundamental Weyl chamber and its orbit has a natural geometric interpretation provided by the next proposition. 

\begin{prop}[{\cite[Prop. 1 and 6]{GaussentLittelmann}}]\label{prop:retraction1}
Let $\rho_{\infty}$ be the retraction with respect to the class $\partial C$ of the antidominant Weyl  chamber. Then  fibers of $\rho_{\infty}:X\to A_0$ are the $U(F)$ orbits on $X$.
Moreover, $\rho_{\infty}$ induces a map $\hat\rho_{\gamma_\lambda}$ from the set of  all galleries of a fixed type $\tau(\gamma_\lambda)$ in $X$ to the galleries of the same type in $A$. Its fibers $C(\delta)= \hat\rho_{\gamma_\lambda}^{-1}(\delta)$, for a fixed gallery $\delta$ of type $\tau(\gamma_\lambda)$ starting at the origin, are locally closed subvarieties isomorphic to an affine space. Their dimension can be computed in terms of dimensions of positively folded galleries.
\end{prop}

Fix a type $\tau$ and denote by $\hat{\mathcal{G}}_\tau$ the minimal galleries in $X$ of type $\tau$ starting at the origin in $\lambda_0$ and write $\mathcal{G}_\tau$ for the set of end-vertices of elements in $\hat{\mathcal{G}}_\tau$.  Then one has the following.

\begin{prop}\label{prop:retractions2}
	Let $\lambda$ be a vertex in the base apartment $A_0$ in $X$ and suppose that the minimal vertex-to-vertex gallery from the origin to $\lambda$ is of type $\tau$. Then 
	\[r_{A_0,c_0}^{-1}(\sW.\lambda)=\mathcal{G}_\tau. \]
	In case $X$ is of algebraic origin the group $K=G(\mathcal{O})$ is the $G$-stabilizer of the origin in $X$ and
	\[r_{A_0,c_0}^{-1}(\sW.\lambda)=K.\lambda\]
\end{prop}

This connection between orbits, galleries and retractions was crucial for the main results in \cite{GaussentLittelmann} but also for the proof of the main theorem in \cite{Convexity} explained in \Cref{sec:Kostant}. 

Similar connections for alcove-to-alcove galleries with Iwahori-subgroup orbits exist (see \cite{PRS} and \cite{MNST}) and we will highlight them in \Cref{sec:ADLV}. 

The connection with shadows is the following. See also Figure~\ref{fig:vertex_shadow} for an example.  

\begin{prop}[Shadows and retractions]\label{prop:shadowsretractions}
	With notation as in \ref{prop:retractions2} suppose that $\lambda$ is an antidominant vertex. Then  
	\[\Shadow^\vee_{\infty}(\lambda)=\rho_{\infty}(r_{A_0,c_0}^{-1}(\lambda)).\]
	In terms of groups one gets
	\[\mu\in\Shadow^\vee_{\infty}(\lambda) \Longleftrightarrow \exists k\in K\text{ and } u\in U(F) \text{ with } uk\mu=\lambda.\] 
\end{prop}
\begin{proof}
	Since $\lambda$ is anti-dominant its $K$-orbit contains the $\sW$-orbit of $\lambda$ inside the base apartment $A_0$. Thus $K.\lambda=r_{A,c_0}^{-1}(\sW.\lambda)$ in this situation. This collection of vertices is, by Proposition~\ref{prop:retractions2}, the set of end-vertices of all minimal galleries inside the building $X$ starting at the origin $\lambda_0$ that are of the same type $\tau$ as a minimal gallery from $\lambda_0$ to $\lambda$. Applying $\rho$ each such gallery gets mapped to a positively folded gallery. Its end-vertex is hence in $\Shadow^\vee_{\infty}(\lambda)$. To see the converse one needs to observe that every positively folded gallery of type $\tau$ unfolds to a minimal gallery in the building by taking the pre-image under ${\rho}_{\infty}$. This follows from Proposition 3.3 in  \cite{Convexity}. 
\end{proof}

\begin{figure}[htb]
	\begin{center}
		\includegraphics[width=0.60\textwidth]{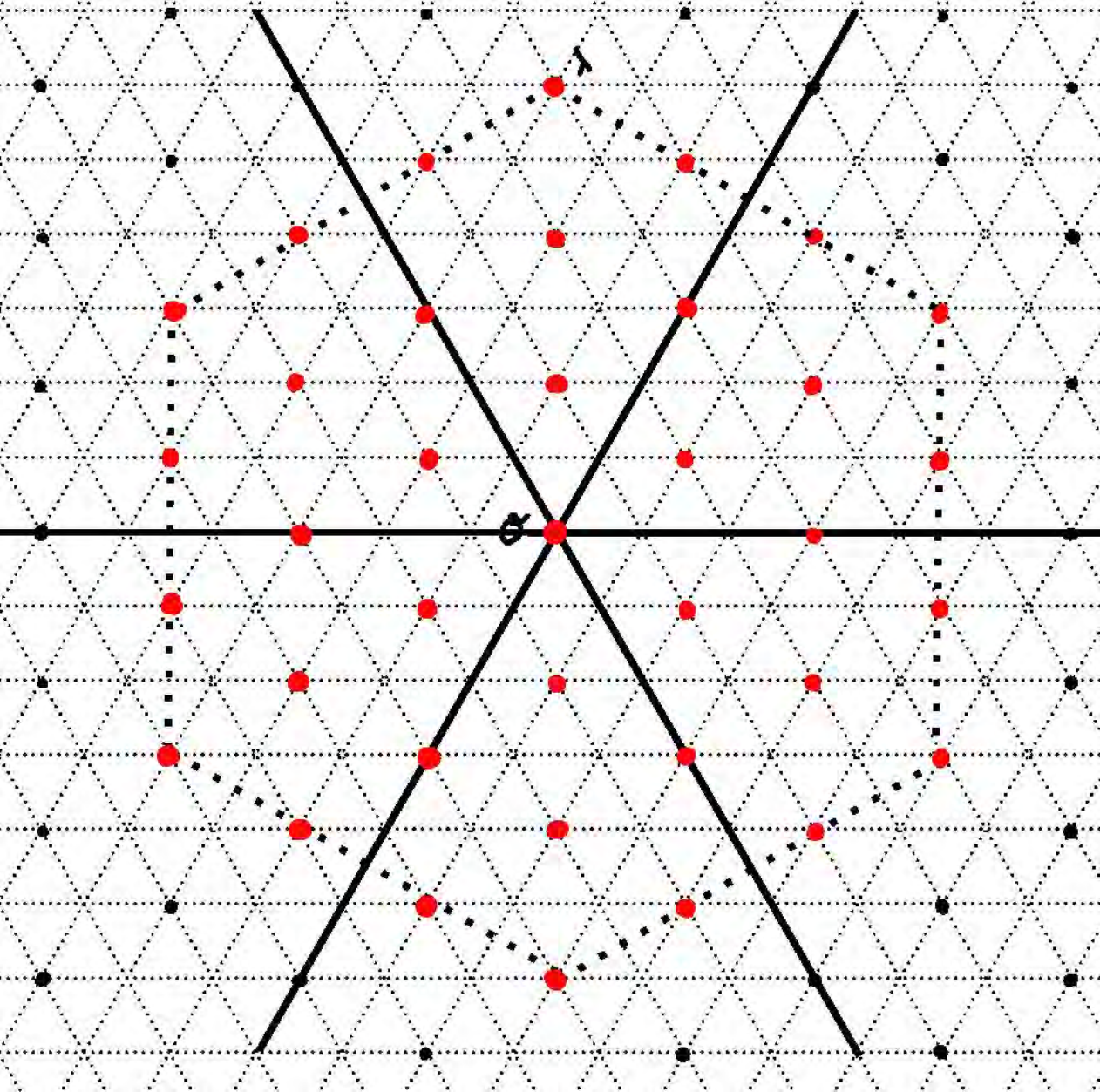}\hspace{1ex}
	\end{center}
\caption{The red vertices are the elements of the vertex shadow of $\lambda$ with respect to the dominant direction (the class of the Weyl chamber containing $\lambda$.}
\label{fig:vertex_shadow}
\end{figure}

Proposition~\ref{prop:shadowsretractions} illustrates that vertex shadows with respect to Weyl chamber orientations arise very naturally as images of balls around the origin under retractions based at infinity. 
It would be interesting to see whether similar characterizations could be shown for shadows with respect to alcove orientations and the retractions based at alcoves.

\section{A first adventure: Kostant convexity }\label{sec:Kostant}

We ended the last section with a statement explaining the red dots in Figure~\ref{fig:vertex_shadow}. There is a second kind of interpretation which is the topic of this section.
Let me continue by telling the algebraic side of the story first. 

Any reductive group $G$ with associated Bruhat-Tits building $X$ decomposes as   
\[G=UTK, \]
where geometrically $U$ is the stabilizer of the parallel class of the antidominant Weyl chamber in the base apartment $A_0$, the group $T$ acts as translations on $A_0$ and $K=G(\mathcal{O})$ is the stabilizer of the origin in $A_0$. The natural projection $pr_T: G\to T$ send a group element $g$ to to its $T$-part, that is the unique element $t$ for which $g=utk$. 

But obviously the subgroup $K$ also acts on the left on $G$. One may hence ask how the $T$-part of a decomposition changes under this left-$K$-action. More precisely one seeks to answer the following question:  

\[\text{For a fixed } t\in T \text{ what is }pr_T(K.t) ?\]  

Fix $t \in T$ and multiply it on the left by some $k\in K$. The product $kt$ itself then can be decomposed as $kt=u't'k'$ for $u'\in U, t'\in T$ and $k'\in K$. Which $t'$ may appear in such a decomposition? 

This question is the natural analog in the context of affine buildings of the classic Kostant convexity theorem for real symmetric spaces. In case of $G=\mathrm{SL}_n(\R)$ similar questions were already studied by Horn and Thompson. The answer for general real Lie groups was given by Kostant in \cite{Kostant}. Historically this result goes back to Schur who showed that the eigenvalues appearing on a diagonal of a matrix are in a convex set. 

Here's what one can prove in algebraic terms. 

\begin{thm}[Convexity theorem, \cite{Convexity}]
\label{thm:KostantConvexity}
With notation as above and $t,t'\in T$ one has 
\[Ut'K\cap KtK\neq \emptyset \Longleftrightarrow t'K\in\conv(\sW.tK). \]
\end{thm} 

The convex hull appearing in this theorem denotes the standard euclidean convex hull which can also described as the minimal intersection of dual half-apartments containing $\sW.tK$. For details see \cite{Convexity}. 

As $K$ stabilizes the origin one can interpret $T$-left cosets of $K$ as vertices in the base apartment. Then the set of vertices of the same type as $tK$ in $\conv(\sW.tK)$ is precisely the vertex shadow of the element $t^+K$ in the orbit $\sW.tK$ that is contained in the fundamental Weyl chamber. The algebraic statement in Theorem~\ref{thm:KostantConvexity} is  a direct consequence of its geometric analog in terms of retractions in an affine building. 

Recall that in Section~\ref{sec:Littelmann} we learned about two kinds of retractions onto a fixed apartment $A_0$ in a building $X$ and how they relate to $U$ and $K$ orbits, folded galleries and shadows. The next theorem explains how the two retractions interact. Theorem~\ref{thm:KostantConvexity} is a direct consequence of~\ref{thm:geometricConvexity}. 

\begin{thm}[Geometric Convexity theorem, \cite{Convexity}]
\label{thm:geometricConvexity}	
	Let $X$ be an affine building in which every panel is contained in at least three alcoves. Fix an apartment $A_0$ in $X$ with origin $\lambda_0$, base alcove $c_0$ and fundamental Weyl chamber $\Cf$ in $A_0$. Let $\sW$ denote the spherical Weyl group associated with $X$ and fix some special vertex $\lambda$ in $A_0$. Then 
	\[\rho_{\Cf, A}(r^{-1}_{c_0, A}(\sW.\lambda))=\conv(\sW.\lambda)\] 
\end{thm}

It is worth mentioning that this geometric result is in fact more general than its algebraic counterpart, since not every affine building is of algebraic origin. 
This geometric version also generalizes to non-discrete euclidean  buildings and more generally $\Lambda$-buildings, where $\Lambda$ is an ordered abelian group. See \cite{Convexity2}.

\begin{figure}[htb]
	\begin{center}
		\includegraphics[width=0.6\textwidth]{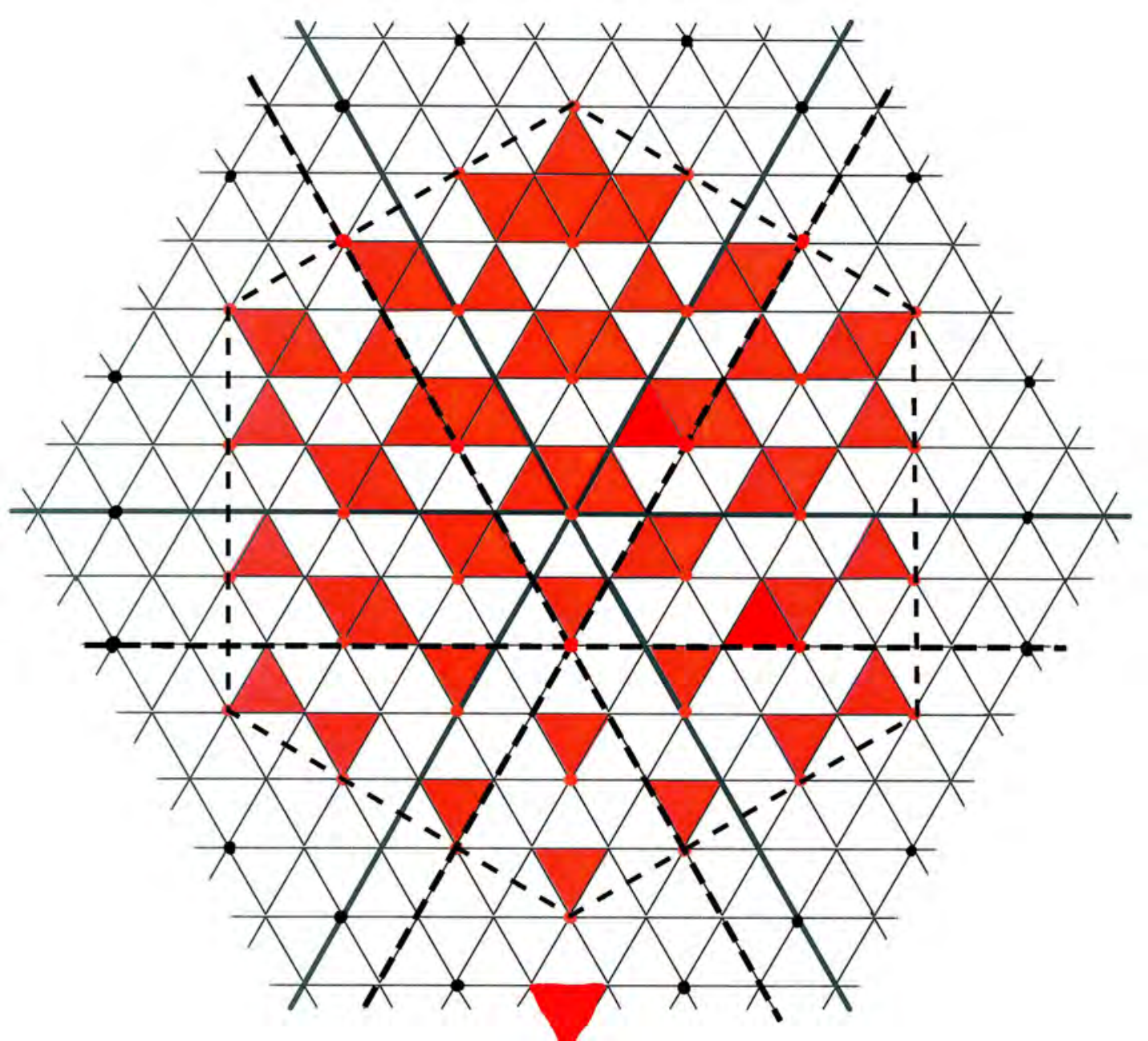}
	\end{center}
	\caption{The shadow of an alcove. }
	\label{fig:YINshadow}
\end{figure}

The connection to galleries becomes clear when one looks at the proof of Theorem~\ref{thm:KostantConvexity}. 	
The two main steps in that proof are the following. 

For simplicity suppose that $tK$ is a dominant vertex in $A_0$ and let $\gamma$ be a minimal gallery from $\lambda_0$ to $tK$. 
On the one hand one needs to show that every vertex in $\conv(\sW.tK)$ is the end-vertex of a positively folded gallery of the same type as $\gamma$. The proof of this statement uses the  character formula for highest weight representations of \cite{GaussentLittelmann} we have highlighted in Section~\ref{sec:Littelmann}. 
Moreover, the endpoint of any positively folded gallery of the same type as $\gamma$ is contained in $\conv(\sW.tK)$. 
Finally one shows that every such positively folded gallery has a minimal preimage under $\rho$. 

It is only natural to ask what happens if one replaces vertices by alcoves. Algebraically this corresponds to considering the stabilizer $I$ of a fundamental alcove rather than the stabilizer $K$ of the origin. One replaces $K$-orbits by $I$-orbits and considers the same question in the affine flag variety $G/I$ rather than the affine Grassmannian $G/K$. Geometrically this corresponds to replacing vertices by alcoves. An example of the patterns that arise from alcove-to-alcove shadows is shown in Figure~\ref{fig:YINshadow}. One has yet to find a closed formula describing the red alcoves in this picture.


\section{A second adventure: Affine Deligne-Lusztig Varieties}\label{sec:ADLV}

If one matches the little ship made of triangles at the top of the red pattern in the alcove shadow shown in Figure~\ref{fig:YINshadow} with one of the dark grey ships appearing close to the black alcove in the grey pattern of Figure~\ref{fig:GHKR} the red alcoves are a subset of the gray alcoves. 

These two pictures hint at a deep connection between folded galleries and the nonemptiness of affine Deligne-Lusztig varieties, for short ADLVs. And indeed, one can determine nonemptiness and compute dimensions of ADLVs in terms of positively folded galleries.

\begin{figure}[htb]
	\centering
		\includegraphics[trim=20 20 20 70 , clip, width=0.5\textwidth]{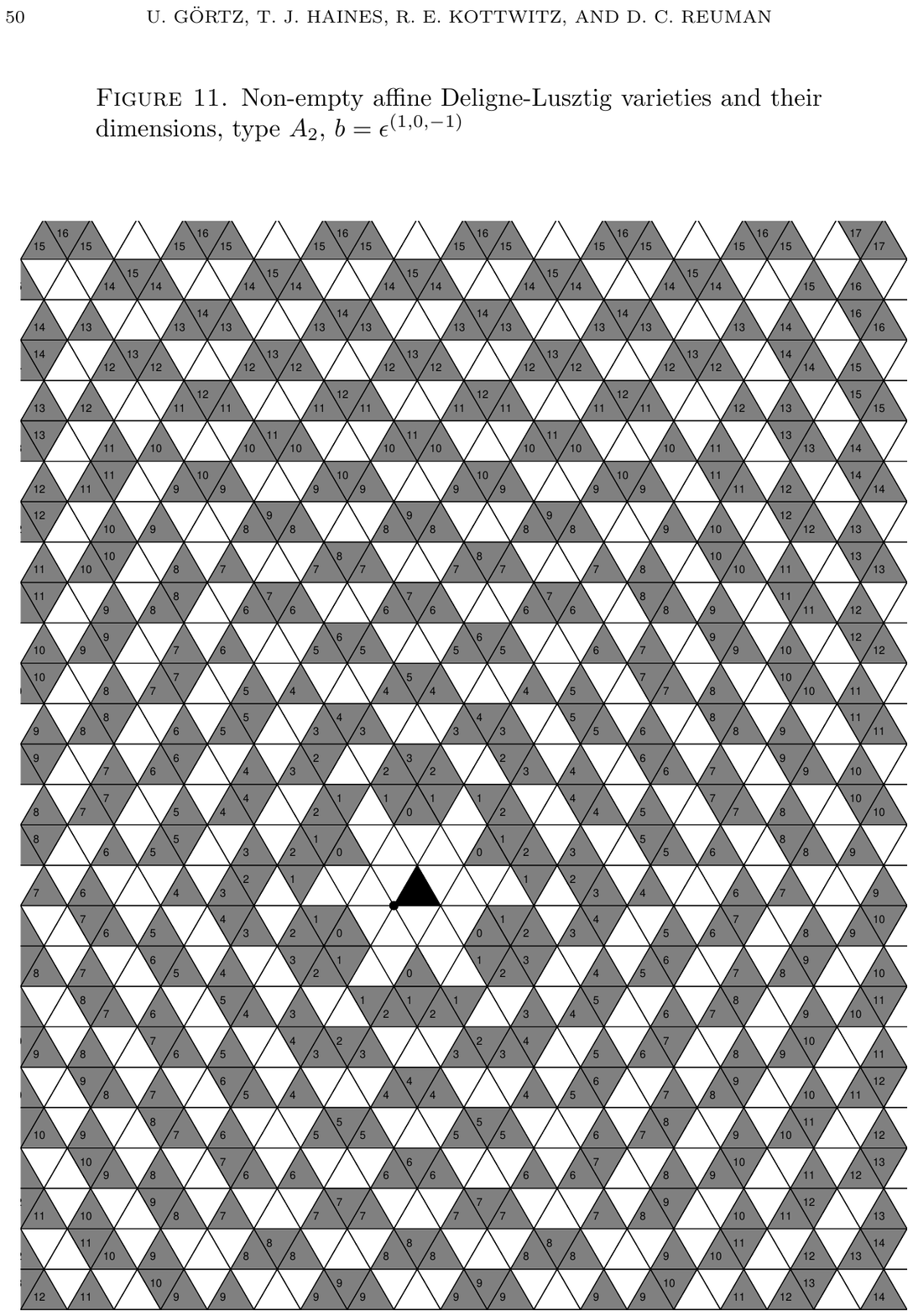}
\caption{Nonempty affine Deligne-Lusztig varieties $X_y(t^\rho)$ represented by the gray-shaded alcoves $y$. Source: \cite{GHKRarxiv}. }
\label{fig:GHKR}
\end{figure}

The formal setup is as follows. 

Let in the following $F= k((t))$ where $k = \overline \F_q$ is the algebraic closure of the finite field of order $q=p^m$ for some prime $p$. Denote by $\sigma$ the Frobenius map of $\F_q$. Then $F$ is a non-archimedian local field with the usual valuation.  The ring of integers will again be denoted by $\cO = k[[t]]$. Remember that there is a natural projection mapping  $\cO \to k$ by setting $t = 0$. This projection detects the constant term $a_0$ in a power series. 

The \emph{affine flag variety} associated with $G$ over $F$ is the quotient $G(F)/I$, where $G$ is a connected reductive group over $\F_q$. Denote again by $B\subset G$ a Borel subgroup containing a maximal torus $T$ and let $I$ denotes the \emph{Iwahori  subgroup} of $G(F)$. Algebraically $I$ is the inverse image of $B(k)$ under the projection $ G(\cO) \to G(k)$.
Geometrically the Iwahori subgroup is the stabilizer of the fundamental alcove in the base apartment $A_0$ in the affine Bruhat-Tits building associated with $G(F)$. 
Similar to the affine Grassmannian case the points of the affine flag variety can be interpreted in the Bruhat-Tits building $X$. As left-cosets of $I$ they correspond bijectively to the maximal simplices, i.e. alcoves in $X$. 
The group $G$ admits the following decomposition: 
\[G(F) = \bigsqcup_{x \in \aW} IxI \hspace{3ex}\text{  Iwahori-Bruhat decomposition.} \]

Here $W$ is the affine Weyl group associated with $G$. 
We are now ready to define affine Deligne-Lusztig varieties inside the affine flag variety. 

\begin{definition}[ADLVs]\label{def:ADLV}
	The \emph{affine Deligne--Lusztig variety}  $X_x(b) \subseteq G(F)/I$ is given by 
	\[ 
	X_x(b) = \{ g \in G(F) \mid g^{-1}b\sigma(g) \in IxI\}/I,  
	\]
	where $x \in \aW, b \in G(F)$. 
\end{definition}

Affine Deligne–Lusztig varieties can be thought of as generalizations of (classical) Deligne–Lusztig varieties to the affine setting. 
Rapoport introduced ADLVs in \cite{RapSatake} in the context of proving Mazur’s theorem. 
The classical Deligne--Lusztig varieties were constructed by Deligne and Lusztig \cite{DL,LuszChev} in order to study the representation theory of finite Chevalley groups. One of the parameters which indexes an
affine Deligne–Lusztig variety is an element of the affine analog of the Weyl group
of a reductive group over a finite field, hence the terminology. Note that ADLVs are in general not affine as varieties. 

One can simplify the situation a bit and also only consider $b\in\aW$.
The reason is that each $b$ in $G(F)$ is $\sigma$-conjugate to an element $b'$ in $\aW$. For $\sigma$-conjugate $b$ and $b'$ one has $X_x(b)\cong X_x(b')$. 
Hence ADLVs are indexed by a pair $x,b$ of elements in the affine Weyl group.   

For a long time the following two very fundamental questions had remained open: 

\begin{itemize}
	\item \textbf{Nonemptiness:} For which $(x,b)\in\aW\times\aW$ is $X_x(b)\neq\emptyset$ ?
	\item \textbf{Dimension:} What is the dimension of $X_x(b)$ ? 
\end{itemize}

The questions of nonemptiness and dimensions of ADLVs may be answered in terms of positively folded galleries and containment conditions of shadows. 
In case $x$ is \emph{basic}, i.e. $x$ admits a fixed point on $\Sigma$, solutions to these questions were known. Later Görtz, He and Nie \cite{GHN} explained the nonemptiness pattern for all $x$ and all basic $b$. Finally He \cite{HeAnnals} and He and Yu \cite{HeYu} gave dimension formulas for all $x$ and (eventually) all $b$. 
The results in \cite{MST} were the first confirming long open conjectures and proving nonemptiness and computing dimensions outside the basic case. 
I will now illustrate this approach which I have developed together with Elizabeth Mili\'cevi\'c and Anne Thomas in \cite{MST, MST2} and also \cite{MNST}. 

The starting point is a result by Görtz, Haines, Kottwitz and Reumann in \cite{GHKR} which characterizes the ADLVs in terms of double coset intersections similar to the ones we had studied in Section~\ref{sec:Kostant} but in the affine flag variety instead of the affine Grassmannian. For this we need to choose orientations on the base apartment $A_0$ in $X$.  

For an element $w\in \sW$ write $\phi_w$ for the Weyl chamber orientation on $A_0$ induced by the Weyl chamber based at $0$ in direction $w$. The corresponding chamber at infinity is fixed by the unipotent radical $^wU^-\define wU^{-}w^{-1}$ of the Borel subgroup $^wB^-\define wB^{-}w^{-1}$. 

\begin{thm}[{\cite[Thm 6.1]{GHKR}}]
	Let $x\in\aW$ and $\lambda$ be a special vertex in $A_0$. Then 
	\begin{enumerate}
		\item $X_x(t^\lambda)\neq\emptyset$ if and only if there exists $w\in\sW$ such that $^wU^-yI\cap IxI\neq\emptyset$.  
		\item if $X_x(t^\lambda)$ is nonempty, then 
		\[\dim(X_x(t^\lambda)=\max\{dim(^wU^-t^{w\lambda}I \cap IxI) -c(\lambda)\}),\]
		where $c$ is a constant depending on $\lambda$. 
	\end{enumerate}
\end{thm}

I have stated here the version of the result for $b$ being a translation. This is the case covered in \cite{MST}. In \cite[Theorem 11.3.1]{GHKRarxiv} a version for arbitrary $X_x(b)$ can be found which we also state as Theorem  2.8 in \cite{MST2}. That result involves intersections of double cosets with $I_P$ orbits on the left with $P$ being a standard (spherical) parabolic subgroup. 
Here $I_P=(I\cap M)N$ for $P=MN$ the Levi decomposition of the standard parabolic. 
These parabolics are closely related to chimney-retractions and shadows with respect to chimneys as studied in \cite{MNST}. 
See also Figure~\ref{fig:AlcoveShadowA2} where the gray strip is a chimney and the colored alcoves form the chimney shadow of the pink alcove labeled $x$. 

The next theorem relates these double cosets to the existence of positively folded galleries and containment of alcoves in shadows with respect to chimney orientations:  

\begin{thm}[\cite{MNST}]\label{thm:MNST}
	Let $x,y,z \in W$ and $P=MN$ a spherical standard parabolic. Let $I_P=(I\cap M)N$ which is a subset of the $P$-chimney stabilizer. Then the following are equivalent
	\begin{enumerate}
		\item $(I_P)^yzI \cap IxI \neq \emptyset $
		\item There exists a $P^y$-positively folded gallery of type $x$ that starts in $\id$ and ends in $z$, that is $x\in \Shadow_{P^y}(z)$. 
	\end{enumerate}
	Moreover the set $\Shadow_{P^y}(z)$ can be computed recursively. 
\end{thm}

An analogous statement holds for intersections of the form $(I_P)^ybK \cap KxK$ and for the analogous double coset intersections with $K$ replaced by some $K_\sigma$ for $\sigma$ any face of the fundamental alcove.  
Putting $I_P=U^-$ in Theorem \ref{thm:MNST} the statement was already shown by Parkinson, Ram and C.Schwer \cite{PRS} as well as by Gaussent and Littelmann in \cite{GaussentLittelmann}.

The key observation here is, that positively folded galleries index non-empty affine Deligne-Lusztig varieties and that their combinatorially defined dimension allows to compute the dimension of the ADLVs. 
In \cite{MST} and \cite{MST2} we construct and manipulate positively folded galleries using delicate combinatorial methods on shadows, folded galleries and generalized root operators to determine which ADLVs are nonempty.

\begin{figure}[ht]
	\begin{center}
		\resizebox{0.5\textwidth}{!}
		{
			\begin{overpic}{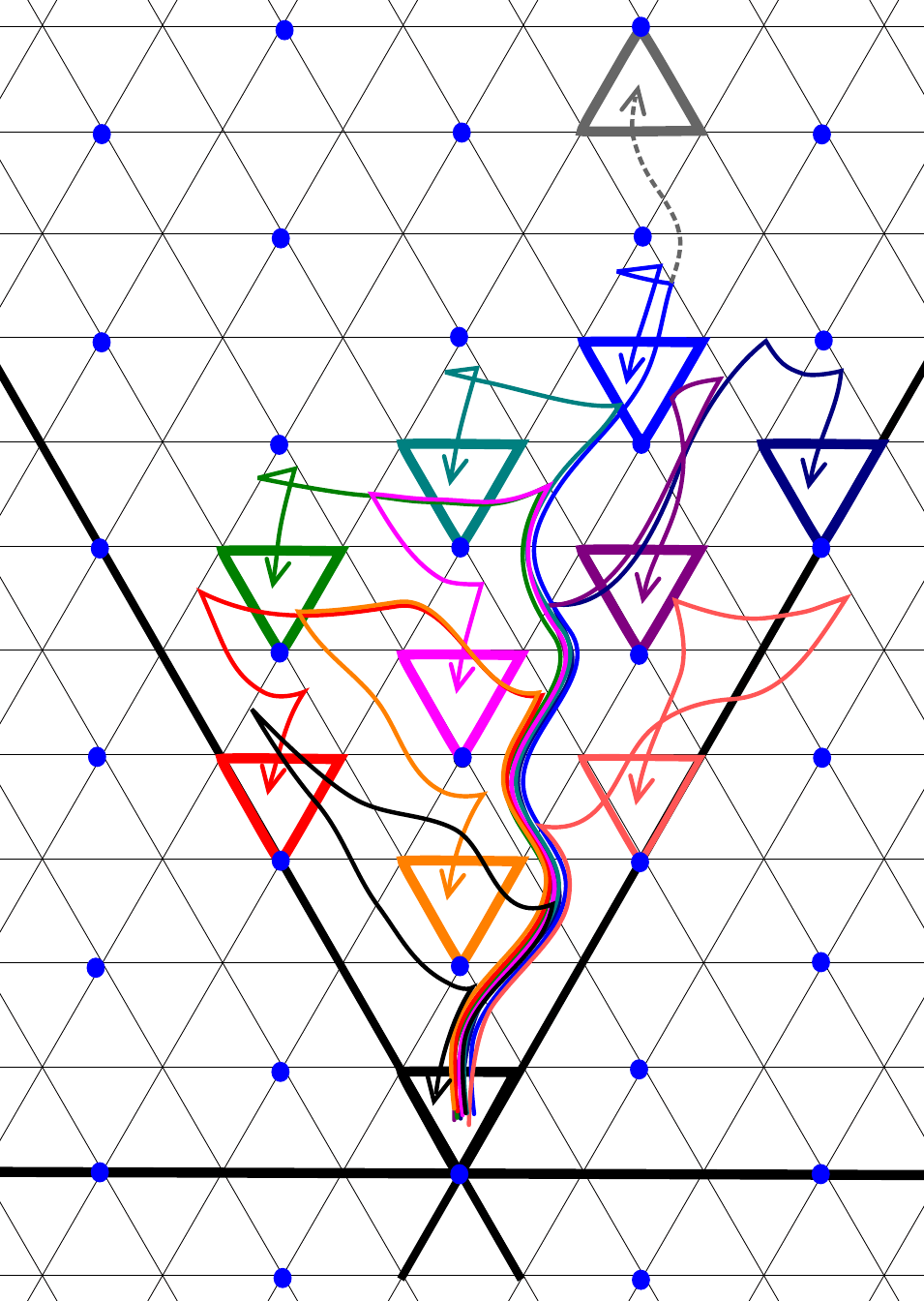}
				\put(50,92){$x$}
			\end{overpic}
		}
	\end{center}
	\caption{$X_x(b) \neq \emptyset$ for most dominant $b=t^\mu$ between 1 and $x$. The nonempty $b$ of that form are outlined with bold color in this figure. This picture is taken from \cite{MST}.  
	}
	\label{fig:galleryADLV}. 
\end{figure}

\begin{example}
	In Figure~\ref{fig:galleryADLV} a family of such positively folded galleries is shown. Each colored gallery corresponds to an alcove $y=t^\lambda\fa$ for some $\lambda$ representing a point in the affine flag variety and simultaneously one inside $X_x(y)$ establishing nonemptiness for this particular variety. The folded galleries are constructed in a very specific way (see Section 6 of \cite{MST}) such that their dimensions  determine the dimension of the corresponding variety. 
	All the folded galleries in the figure are of the same type as the gray minimal connecting the fundamental alcove with $x$. One of the colored galleries is obtained from any other by a finite sequence of applications of root operators. The topmost gallery (blue) is explicitly constructed to satisfy the neccessary folding and dimension criteria. 
\end{example}

Our approach in \cite{MST, MST2} can be summarized as follows: 

\begin{itemize}
		\item[(1)] An ADLV $X_x(b)$ is nonempty if and only if there exists a positively folded gallery from $\id$ to (a $\sigma$-conjugate of) $b$ of type $w$ with $[w]=x$. 
		\item[(2)] The dimension $dim(X_x(b))$ can be computed via the dimension of these galleries. 		
		\item[(3)] Construct and manipulate such galleries using root operators, combinatorics in Coxeter complexes and explicit transformations.  
\end{itemize}

We close this section with an example of such an explicit  construction  and one example of a theorem we are able to prove using this approach.
Figure~\ref{fig:galleryconstruction} shows how to construct a gallery from the fundamental alcove to itself of type $a = t^{2\rho}w_0$. Its existence proves nonemptiness of the ADLV $X_a(\id)$. 
	
\begin{figure}[htb]
		\begin{center}
			\begin{overpic}[trim=0 0 0 5, clip,  width=0.63\textwidth]{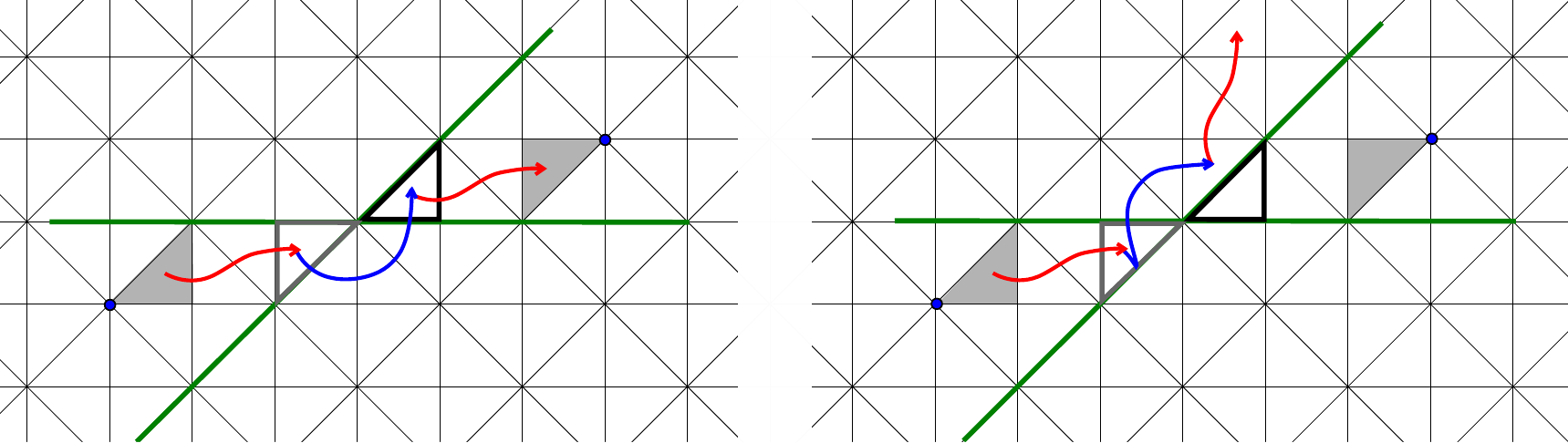} 
				\put(34,18){$a$}
				\put(8,10){$1$}
			\end{overpic}\hspace{2ex}
			\includegraphics[trim=0 5 0 3, clip,  width=0.32\textwidth]{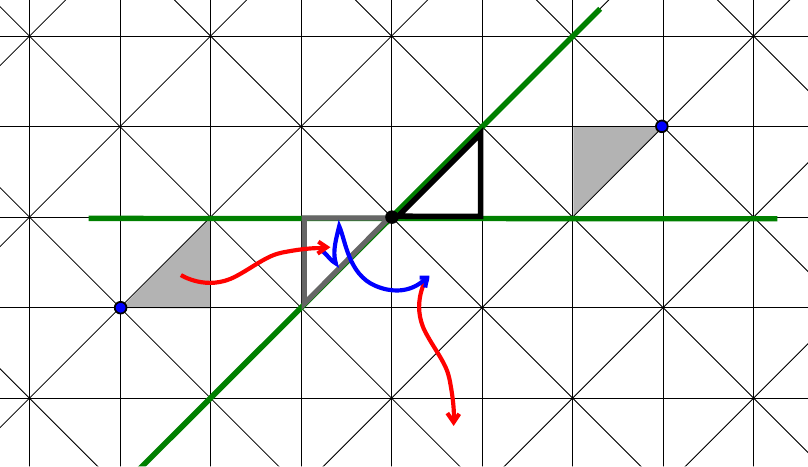}\newline
			
			\includegraphics[width=0.32\textwidth]{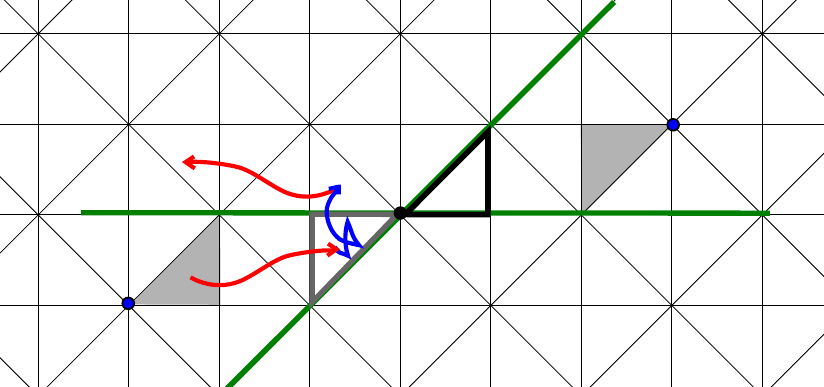} \hspace{2ex}
			\includegraphics[width=0.32\textwidth]{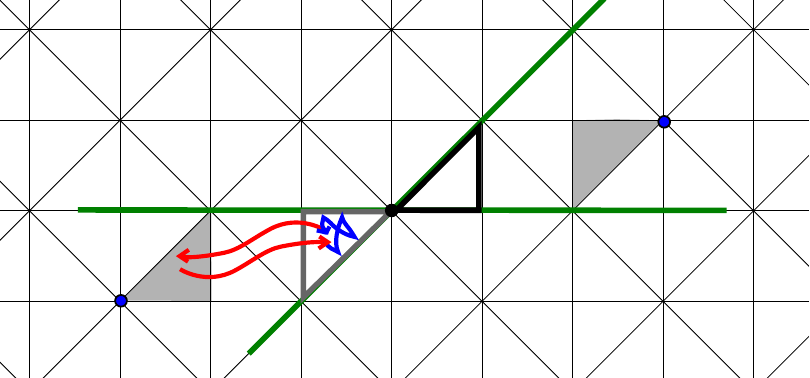} \hspace{2ex}
			
		\end{center}
	\caption{This construction proves that  $X_a(1) \neq \emptyset$ and of dimension $\geq 7$.}
	\label{fig:galleryconstruction}
\end{figure}

\begin{thm}[\cite{MST}]\label{thm:MST}
		Let $b = t^\mu$ be a pure translation and let $x = t^\lambda w \in W$.
		Assume that  $b$ is in the convex hull of $x$ and the base alcove \newline
		+ two technical conditions on $\mu$ and $\lambda$. 
		Then
		\[ X_x(1) \neq \emptyset \implies X_x(b) \neq \emptyset \]
		and if $w = w_0$ then $X_x(1) \neq \emptyset$ and $X_x(b) \neq \emptyset$.
		\smallskip
		
		If both varieties are nonempty then
		\[ \dim X_x(b) = \dim X_x(1) - \langle \rho, \mu^+ \rangle.\]  
\end{thm}

Theorem~\ref{thm:MST} relates nonemptiness and the dimension count of $X_x(b)$ for an arbitrary translations $b=t^\mu$ to nonemptiness, respectively the dimension of $X_x(\id)$. See \cite{MST, MST2} for many other results in that direction.


\section{It's a wild world}\label{sec:final}

I hope it has become clear that folded galleries and shadows are interesting for at least two reasons. For one they give rise to rich combinatorial structures in Coxeter groups which are not fully understood. On the other hand they provide a powerful tool to study varieties associated with both the affine Grassmannian and the affine flag variety. 

\subsection{More applications}\label{sec:applications} 

There are many more applications than the ones highlighted in this work. One direction that has been active is the study of structure constants and symmetric functions in the context of Hecke algebras. I would like to restrict myself to very few, non-exhaustive comments here. There are probably many references I have missed. 

An equivalent formulation of the gallery model was given by Ram \cite{Ram} using alcove walks.  Ram also shows in \cite{Ram} that the affine Hecke algebra can be modeled as a quotient of the algebra of alcove walks with respect to a small list of relations. The product of elements in the alcove walk algebra is simply the concatenation of alcove walks. 

The connection between symmetric functions, crystal theory, Hecke algebras and the path model and its formulation in terms of alcove walks is explained in detail in \cite{Ram}. 
The results highlighted in \cite{Ram} show that the theory of folded galleries is a powerful tool to understand structure constant and $q$-analogs of crystal structures. 

Using the combinatorics of folded galleries, or alcove walks, one can explain base change formulas in the affine Hecke algebra \cite{GriffethRam}.  
C. Schwer used folded galleries to compute Hall--Littlewood polynomials in \cite{CSchwer}. 
Ram and Yip study McDonald Polynomials using alcove walks in  \cite{RamYip}.

Let me highlight one other direction of study, namely a gallery interpretation of Mirkovi\'c -Vilonen cycles, MV-cycles for short, and MV-polytopes. This connection was already established in \cite{GaussentLittelmann}, where the authors provide a representation theoretic interpretation of the combinatorial character formula stated as Theorem~\ref{thm:character} above. 

Denote by $X_\lambda^\mu$ the closure of $U^-(F).\mu\cap K.\lambda$.  
This set may be expressed as a union of endpoints of positively folded galleries in the projection of a cell $Z(\delta)$ corresponding to an LS-gallery $\delta$ of type $\lambda$ with endpoint $\mu$. See page 38 in \cite{GaussentLittelmann}  for a precise statement. 

\begin{thm}[{\cite[Thm C]{GaussentLittelmann}}]
The cells $Z(\delta)$ indexed by the LS galleries of type $\lambda$ and endpoint $\mu$ are exactly the irreducible components of $X_\lambda^\mu$.  These irreducible components are exactly the MV cycles. 
\end{thm}	

Mirkovi\'c--Vilonen cycles are a family of subvarieties of the affine Grassmannian with the following property: Denote by $Gr^\lambda$ the subvariety whose  intersection homology is isomorphic to the irreducible representation $V(\lambda)$ of $G^\vee$ of highest weight $\lambda$. The subset of the MV cycles lying in $Gr^\lambda$ forms a basis for the intersection homology. This implies that there is also a basis of $V(\lambda)$ indexed by MV cycles. 

MV-polytopes are images of MV-cycles $C_{\nu, \lambda^+}$ under the moment map with respect to the $T$ action on the affine Grassmannian.  
The $C_{\nu, \lambda^+}$ are closed $T$-invariant subvarieties of $G/K$. 
	
Ehrig \cite{Ehrig} studies MV-polytopes by means of Bruhat-Tits buildings and gives a type-independent characterization of MV-polytopes by assigning to every LS-gallery an explicitly constructed MV-polytope.

Suppose a vertex $\nu$ is given and let $\gamma$ be an LS-gallery of type $\lambda$ with end-vertex $\nu$, then the following theorem holds. 

\begin{thm}[\cite{Ehrig}]
The MV polytope associated with $\nu, \lambda$ is the convex hull of certain other LS-galleries of type $\lambda$.
\end{thm}	

One can in fact prove that the MV polytope is also the same as the convex hull of all (partial) unfoldings of $\gamma$. 
An example of such a polytope is given in Figure~\ref{fig:MVcycle}.  

\begin{figure}[htb]
	\centering
	\includegraphics[width=0.5\textwidth]{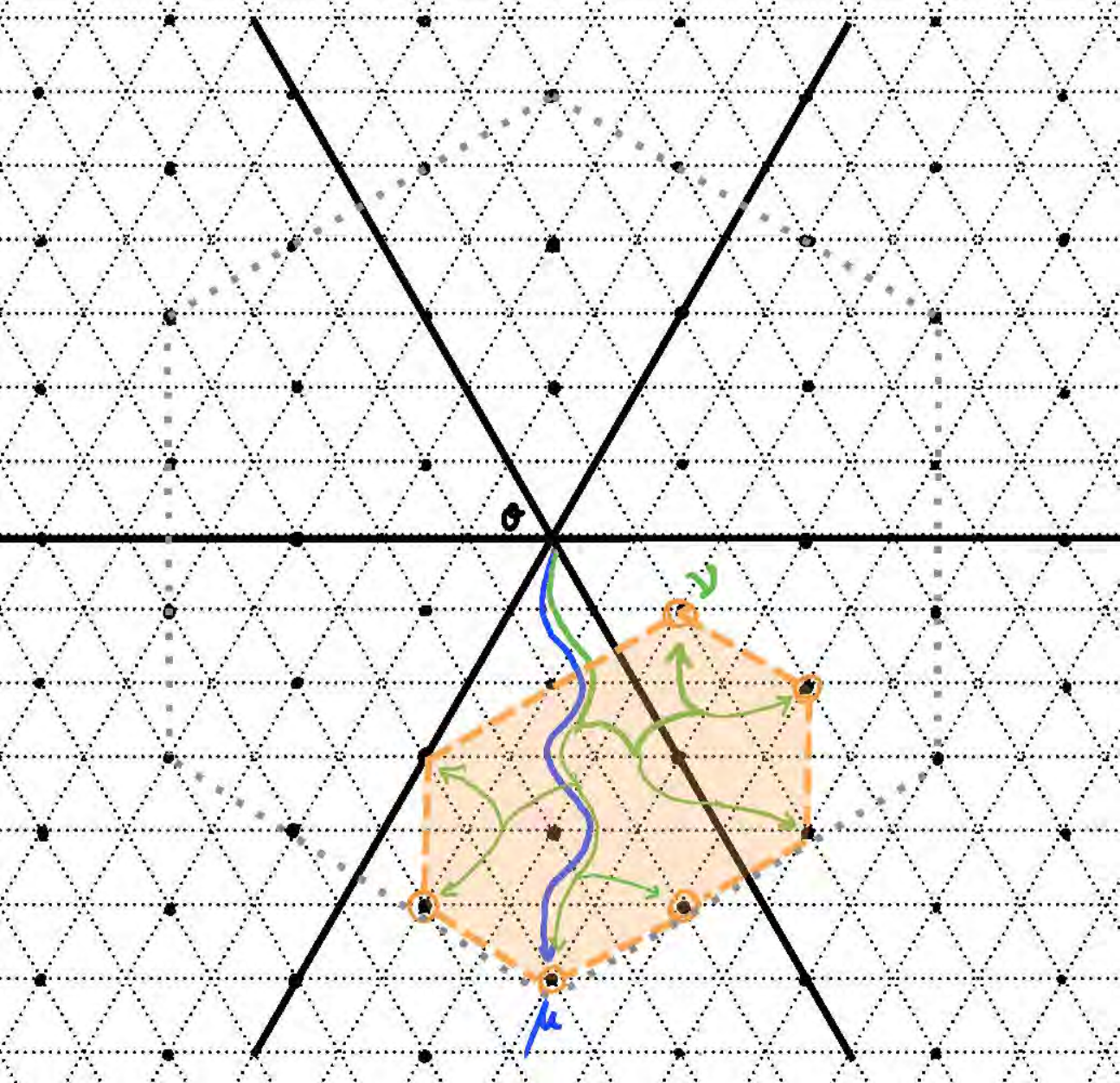}\hspace{2ex}
	\caption{The light-orange set is the MV-polytope associated with $\lambda$ and $\nu$. }
	\label{fig:MVcycle} 
\end{figure}

\subsection{What's next? }

Throughout the text I have already hinted at open problems and questions here and there. Let me collect these and some more  for future reference.

\begin{enumerate}
	\item So far only recursive descriptions of alcove shadows are known. Provide closed formulas for alcove and Weyl chamber shadows in (affine) Coxeter groups. 
	\item In Theorem~\ref{thm:regular_shadow} we provide a recursive description of shadows with respect to Weyl chamber orientations. It would be interesting to study the analogous behavior of shadows in other classes of Coxeter groups or with respect to other braid-invariant orientations. 
	\item Introduce braid-invariant orientations in hyperbolic Coxeter groups and compute (either via closed formulas or recursively) their shadows. 
	\item Proposition~\ref{prop:shadowsretractions} illustrates that vertex shadows with respect to Weyl chamber orientations arise very naturally as images of balls around the origin under retractions based at infinity. 
	It would be interesting to see whether similar characterizations could be shown for shadows with respect to alcove orientations and the retractions based at alcoves. 

	\item Kamnitzer \cite{KamnitzerAnnals} gives an explicit description of MV cycles and describes MV polytopes as pseudo-Weyl polytopes whose $2$-faces are rank 2 MV polytopes themselves. This suggests that shadows should also be determined by the shadows in rank two. 
	

	\item Motivated by the connection with Bruhat order I would be interested to see what the general connection between order relations on Coxeter groups and shadows is. 
	  
	\item \cite{MST, MST2} do not show nonemptiness of ADLVs in full generality. Complete the proof of nonemptiness and the dimension count of ADLV in terms of folded galleries. 
	\item Study the relationship between sub-varieties of the affine Grassmannian and affine flag variety and vertex, respectively alcove shadows. What is the connection to moment maps?
	\item Has the precise connection between positively folded galleries, the results of Lusztig \cite{Lusztig} and the work of  Braverman and Gaitsgory \cite{BravermanGaitsgory} been clarified? This question was mentioned in \cite{GaussentLittelmann}.

\end{enumerate}

There are, as always, many more questions that remain open. And I am curious to see what other results and applications will appear on this topic in the future.

\subsection*{Acknowledgments}
This work was supported through the program \emph{Research in Pairs}  by the Mathematisches Forschungsinstitut Oberwolfach in 2021. It was, as always, a great pleasure to work at MFO's excellent library.  
Many thanks also to my coauthors Elizabeth Mili\'cevi\'c and Anne Thomas. Parts of this survey are based on our joint work and the way I think about the material presented here is definitely influenced by our collaboration.   
In addition I would like to thank the anonymous referee for his or her detailed and thoughtful report and St\'ephane Gaussent, Elizabeth Mili\'cevi\'c, Anne Thomas and Linus Kramer for their helpful comments on an earlier version of this survey.

\nocite{BBirula, Beazley, GaussentLittelmann2, LenartPostnikov, Littelmann1, Littelmann2, Matsumoto, PRS, Ram, KM, LusztigCanonical, RapSatake, LakshmibaiSeshadri, HeNoteOn}
\renewcommand{\refname}{Bibliography}
\bibliography{Bibliography}

\begin{thebibliography}{GHKR10}

\bibitem[AB08]{AbramenkoBrown}
Peter Abramenko and Kenneth~S. Brown.
\newblock {\em Buildings}, volume 248 of {\em Graduate Texts in Mathematics}.
\newblock Springer, New York, 2008.
\newblock Theory and applications.

\bibitem[BB05]{BjoernerBrenti}
Anders Bj\"orner and Francesco Brenti.
\newblock {\em Combinatorics of {C}oxeter groups}, volume 231 of {\em Graduate
  Texts in Mathematics}.
\newblock Springer, New York, 2005.

\bibitem[Bea12]{Beazley}
Elizabeth Beazley.
\newblock Affine {D}eligne-{L}usztig varieties associated to additive affine
  {W}eyl group elements.
\newblock {\em J. Algebra}, 349:63--79, 2012.

\bibitem[BG01]{BravermanGaitsgory}
Alexander Braverman and Dennis Gaitsgory.
\newblock Crystals via the affine {G}rassmannian.
\newblock {\em Duke Math. J.}, 107(3):561--575, 2001.

\bibitem[BH99]{BH}
Martin~R. Bridson and Andr\'{e} Haefliger.
\newblock {\em Metric spaces of non-positive curvature}, volume 319 of {\em
  Grundlehren der Mathematischen Wissenschaften}.
\newblock Springer-Verlag, Berlin, 1999.

\bibitem[{Bia}73]{BBirula}
Andrzej {Bialynicki-Birula}.
\newblock {Some theorems on actions of algebraic groups}.
\newblock {\em {Ann. Math. (2)}}, 98:480--497, 1973.

\bibitem[BK12]{BerensteinKapovich}
Arkady Berenstein and Michael Kapovich.
\newblock Affine buildings for dihedral groups.
\newblock {\em Geom. Dedicata}, 156:171--207, 2012.

\bibitem[BMPS19]{reflectionlength}
Joel {Brewster Lewis}, Jon {McCammond}, T.~Kyle {Petersen}, and Petra {Schwer}.
\newblock {Computing reflection length in an affine Coxeter group}.
\newblock {\em {Trans. Am. Math. Soc.}}, 371(6):4097--4127, 2019.

\bibitem[{Bou}08]{Bourbaki}
Nicolas {Bourbaki}.
\newblock {\em {Elements of mathematics. Lie groups and Lie algebras. Chapters
  4--6.}}
\newblock Berlin: Springer, 2008.

\bibitem[Bro89]{Brown}
Kenneth~S. Brown.
\newblock {\em Buildings}.
\newblock Springer-Verlag, New York, 1989.

\bibitem[Dav08]{Davis}
Michael~W. Davis.
\newblock {\em The geometry and topology of {C}oxeter groups}, volume~32 of
  {\em London Mathematical Society Monographs Series}.
\newblock Princeton University Press, Princeton, NJ, 2008.

\bibitem[Dav09]{DavisConstruction}
Michael~W. Davis.
\newblock Examples of buildings constructed via covering spaces.
\newblock {\em Groups Geom. Dyn.}, 3(2):279--298, 2009.

\bibitem[Deo85]{Deodhar}
Vinay~V. Deodhar.
\newblock On some geometric aspects of {B}ruhat orderings. {I}. {A} finer
  decomposition of {B}ruhat cells.
\newblock {\em Invent. Math.}, 79(3):499--511, 1985.

\bibitem[DL76]{DL}
Pierre Deligne and George Lusztig.
\newblock Representations of reductive groups over finite fields.
\newblock {\em Ann. of Math. (2)}, 103(1):103--161, 1976.

\bibitem[Ehr10]{Ehrig}
Michael Ehrig.
\newblock M{V}-polytopes via affine buildings.
\newblock {\em Duke Math. J.}, 155(3):433--482, 2010.

\bibitem[FS95]{FunkStrambach}
Martin Funk and Karl Strambach.
\newblock Free constructions.
\newblock In {\em Handbook of incidence geometry}, pages 739--780.
  North-Holland, Amsterdam, 1995.

\bibitem[Gar97]{Garrett}
Paul Garrett.
\newblock {\em Buildings and classical groups}.
\newblock Chapman \& Hall, London, 1997.

\bibitem[GHKR06]{GHKR}
Ulrich Görtz, Thomas~J. Haines, Robert~E. Kottwitz, and Daniel~C. Reuman.
\newblock Dimensions of some affine {D}eligne-{L}usztig varieties.
\newblock {\em Ann. Sci. \'Ecole Norm. Sup. (4)}, 39(3):467--511, 2006.

\bibitem[GHKR10]{GHKRarxiv}
Ulrich Görtz, Thomas~J. Haines, Robert~E. Kottwitz, and Daniel~C. Reuman.
\newblock Affine deligne–lusztig varieties in affine flag varieties.
\newblock {\em Compositio Mathematica}, 146(5):1339–1382, Jul 2010.

\bibitem[GHN15]{GHN}
Ulrich {G\"ortz}, Xuhua {He}, and Sian {Nie}.
\newblock {\(\mathbf{P}\)-alcoves and nonemptiness of affine Deligne-Lusztig
  varieties. }.
\newblock {\em {Ann. Sci. \'Ec. Norm. Sup\'er. (4)}}, 48(3):647--665, 2015.

\bibitem[GL05]{GaussentLittelmann}
Stéphane Gaussent and Peter Littelmann.
\newblock L{S} galleries, the path model, and {MV} cycles.
\newblock {\em Duke Math. J.}, 127(1):35--88, 2005.

\bibitem[GL12]{GaussentLittelmann2}
St\'ephane Gaussent and Peter Littelmann.
\newblock One-skeleton galleries, the path model, and a generalization of
  {M}acdonald's formula for {H}all-{L}ittlewood polynomials.
\newblock {\em Int. Math. Res. Not. IMRN}, (12):2649--2707, 2012.

\bibitem[GR04]{GriffethRam}
Stephen Griffeth and Arun Ram.
\newblock Affine {H}ecke algebras and the {S}chubert calculus.
\newblock {\em European J. Combin.}, 25(8):1263--1283, 2004.

\bibitem[GS20]{Shadows}
Marius Graeber and Petra Schwer.
\newblock Shadows in {C}oxeter groups.
\newblock {\em Ann. Comb.}, 24(1):119--147, 2020.

\bibitem[He13]{HeNoteOn}
Xuhua He.
\newblock Note on affine {D}eligne-{L}usztig varieties.
\newblock 2013.
\newblock arXiv:1309.0075 [math.AG].

\bibitem[{He}14]{HeAnnals}
Xuhua {He}.
\newblock Geometric and homological properties of affine {D}eligne-{L}usztig
  varieties.
\newblock {\em Ann. of Math. (2)}, 179(1):367--404, 2014.

\bibitem[Hit10]{Convexity}
Petra Hitzelberger.
\newblock Kostant convexity for affine buildings.
\newblock {\em Forum Math.}, 22(5):959--971, 2010.

\bibitem[Hit11]{Convexity2}
Petra Hitzelberger.
\newblock Non-discrete affine buildings and convexity.
\newblock {\em Adv. Math.}, 227(1):210--244, 2011.

\bibitem[Hum90]{Humphreys}
James~E. Humphreys.
\newblock {\em Reflection groups and {C}oxeter groups}, volume~29 of {\em
  Cambridge Studies in Advanced Mathematics}.
\newblock Cambridge University Press, Cambridge, 1990.

\bibitem[HY21]{HeYu}
Xuhua {He} and Qingchao {Yu}.
\newblock {Dimension formula for the affine Deligne-Lusztig variety \(X(\mu ,
  b)\)}.
\newblock {\em {Math. Ann.}}, 379(3-4):1747--1765, 2021.

\bibitem[Kam10]{KamnitzerAnnals}
Joel Kamnitzer.
\newblock Mirkovi\'{c}-{V}ilonen cycles and polytopes.
\newblock {\em Ann. of Math. (2)}, 171(1):245--294, 2010.

\bibitem[{Kas}90]{Kashiwara}
Masaki {Kashiwara}.
\newblock {Crystalizing the q-analogue of universal enveloping algebras}.
\newblock {\em {Commun. Math. Phys.}}, 133(2):249--260, 1990.

\bibitem[KM08]{KM}
Michael Kapovich and John~J. Millson.
\newblock A path model for geodesics in {E}uclidean buildings and its
  applications to representation theory.
\newblock {\em Groups Geom. Dyn.}, 2(3):405--480, 2008.

\bibitem[{Kos}73]{Kostant}
Bertram {Kostant}.
\newblock {On convexity, the Weyl group and the Iwasawa decomposition}.
\newblock {\em {Ann. Sci. \'Ec. Norm. Sup\'er. (4)}}, 6:413--455, 1973.

\bibitem[Lit94]{Littelmann1}
Peter Littelmann.
\newblock A {L}ittlewood-{R}ichardson rule for symmetrizable {K}ac-{M}oody
  algebras.
\newblock {\em Invent. Math.}, 116(1-3):329--346, 1994.

\bibitem[Lit95]{Littelmann2}
Peter Littelmann.
\newblock Paths and root operators in representation theory.
\newblock {\em Ann. of Math. (2)}, 142(3):499--525, 1995.

\bibitem[LP08]{LenartPostnikov}
Cristian Lenart and Alexander Postnikov.
\newblock A combinatorial model for crystals of {K}ac-{M}oody algebras.
\newblock {\em Trans. Amer. Math. Soc.}, 360(8):4349--4381, 2008.

\bibitem[LS91]{LakshmibaiSeshadri}
Venkatramani Lakshmibai and Conjeeveram~S. Seshadri.
\newblock Standard monomial theory.
\newblock In {\em Proceedings of the {H}yderabad {C}onference on {A}lgebraic
  {G}roups ({H}yderabad, 1989)}, pages 279--322. Manoj Prakashan, Madras, 1991.

\bibitem[Lus78]{LuszChev}
George Lusztig.
\newblock {\em Representations of finite {C}hevalley groups}, volume~39 of {\em
  CBMS Regional Conference Series in Mathematics}.
\newblock American Mathematical Society, Providence, R.I., 1978.
\newblock Expository lectures from the CBMS Regional Conference held at
  Madison, Wis., August 8--12, 1977.

\bibitem[Lus90]{LusztigCanonical}
G.~Lusztig.
\newblock Canonical bases arising from quantized enveloping algebras.
\newblock {\em J. Amer. Math. Soc.}, 3(2):447--498, 1990.

\bibitem[Lus96]{Lusztig}
George Lusztig.
\newblock An algebraic-geometric parametrization of the canonical basis.
\newblock {\em Adv. Math.}, 120(1):173--190, 1996.

\bibitem[Mat64]{Matsumoto}
Hideya Matsumoto.
\newblock G\'en\'erateurs et relations des groupes de {W}eyl g\'en\'eralis\'es.
\newblock {\em C. R. Acad. Sci. Paris}, 258:3419--3422, 1964.

\bibitem[MNST20]{MNST}
Elizabeth Milićević, Yusra Naqvi, Petra Schwer, and Anne Thomas.
\newblock A gallery model for affine flag varieties via chimney retractions.
\newblock 2020.
\newblock arXiv:1912.09911 [math.RT], 39 pages, 2020.

\bibitem[MST19]{MST}
Elizabeth Mili\'{c}evi\'{c}, Petra Schwer, and Anne Thomas.
\newblock Dimensions of affine {D}eligne-{L}usztig varieties: a new approach
  via labeled folded alcove walks and root operators.
\newblock {\em Mem. Amer. Math. Soc.}, 261(1260):v+101, 2019.

\bibitem[MST20]{MST2}
Elizabeth Milićević, Petra Schwer, and Anne Thomas.
\newblock Affine {D}eligne-{L}usztig varieties and folded galleries governed by
  chimneys.
\newblock 2020.
\newblock arXiv:2006.16288 [math.AG], 53 pages, 2020.

\bibitem[PRS09]{PRS}
James Parkinson, Arun Ram, and Christoph Schwer.
\newblock Combinatorics in affine flag varieties.
\newblock {\em J. Algebra}, 321(11):3469--3493, 2009.

\bibitem[Ram06]{Ram}
Arun Ram.
\newblock Alcove walks, {H}ecke algebras, spherical functions, crystals and
  column strict tableaux.
\newblock {\em Pure Appl. Math. Q.}, 2(4, Special Issue: In honor of Robert D.
  MacPherson. Part 2):963--1013, 2006.

\bibitem[Rap00]{RapSatake}
M.~Rapoport.
\newblock A positivity property of the {S}atake isomorphism.
\newblock {\em Manuscripta Math.}, 101(2):153--166, 2000.

\bibitem[{Ron}86]{Ronan-free1}
Mark {Ronan}.
\newblock {A construction of buildings with no rank 3 residues of spherical
  type}.
\newblock {Buildings and the geometry of diagrams, Lect. 3rd 1984 Sess.
  C.I.M.E., Como/Italy 1984, Lect. Notes Math. 1181, 242-248 (1986).}, 1986.

\bibitem[Ron09a]{RonanConstruction}
Mark Ronan.
\newblock Construction and automorphisms of buildings.
\newblock In {\em Essays in geometric group theory}, volume~9 of {\em Ramanujan
  Math. Soc. Lect. Notes Ser.}, pages 125--143. Ramanujan Math. Soc., Mysore,
  2009.

\bibitem[Ron09b]{Ronan}
Mark Ronan.
\newblock {\em Lectures on buildings}.
\newblock University of Chicago Press, Chicago, IL, 2009.
\newblock Updated and revised.

\bibitem[RT87]{RonanTits-free2}
Mark {Ronan} and Jacques {Tits}.
\newblock {Building buildings}.
\newblock {\em {Math. Ann.}}, 278:291--306, 1987.

\bibitem[RY11]{RamYip}
Arun Ram and Martha Yip.
\newblock A combinatorial formula for {M}acdonald polynomials.
\newblock {\em Adv. Math.}, 226(1):309--331, 2011.

\bibitem[Sch06]{CSchwer}
Christoph Schwer.
\newblock Galleries, {H}all-{L}ittlewood polynomials, and structure constants
  of the spherical {H}ecke algebra.
\newblock {\em Int. Math. Res. Not.}, pages Art. ID 75395, 31, 2006.

\bibitem[Ser03]{SerreTrees}
Jean-Pierre Serre.
\newblock {\em Trees}.
\newblock Springer Monographs in Mathematics. Springer-Verlag, Berlin, 2003.
\newblock Translated from the French original by John Stillwell, Corrected 2nd
  printing of the 1980 English translation.

\bibitem[Tit60]{Tits-Groupes}
Jacques Tits.
\newblock Groupes et géométries de coxeter.
\newblock {\em IHES Notes polycopies 1960}, 1960.

\bibitem[{Tit}86]{TitsComo}
Jacques {Tits}.
\newblock {Immeubles de type affine. (Buildings of affine type)}.
\newblock {Buildings and the geometry of diagrams, Lect. 3rd 1984 Sess.
  C.I.M.E., Como/Italy 1984, Lect. Notes Math. 1181, 159-190 (1986).}, 1986.

\bibitem[Wei09]{WeissAffine}
Richard~M. Weiss.
\newblock {\em The structure of affine buildings}, volume 168 of {\em Annals of
  Mathematics Studies}.
\newblock Princeton University Press, Princeton, NJ, 2009.

\bibitem[{Zhu}17]{Zhu}
Xinwen {Zhu}.
\newblock {An introduction to affine Grassmannians and the geometric Satake
  equivalence}.
\newblock In {\em Geometry of moduli spaces and representation theory. Lecture
  notes from the 2015 IAS/Park City Mathematics Institute (PCMI) Graduate
  Summer School, Park City, UT, USA, June 28 -- July 18, 2015}, pages 59--154.
  Providence, RI: American Mathematical Society (AMS); Princeton, NJ: Institute
  for Advanced Study (IAS), 2017.

\end{thebibliography}
\bibliographystyle{alpha}

\end{document}